\documentclass[12pt]{article}
\usepackage{amsmath}
\usepackage{amsbsy}
\usepackage{graphicx}
\usepackage{enumerate}
\usepackage{natbib}
\usepackage{url} 
\usepackage{mathrsfs}
\usepackage{hyperref}
\usepackage{color} 
\usepackage{bm} 
\usepackage{amsthm} 
\usepackage{amssymb} 
\usepackage{verbatim} 
\usepackage{enumitem} 
\usepackage{epsfig} 
\usepackage{fullpage}
\usepackage{setspace}
\usepackage{rotating} 
\usepackage{ulem}   

\usepackage{array}
\usepackage{booktabs} 
\usepackage{caption}
\usepackage{tabularx}
\usepackage{float}
\usepackage{graphicx}
\usepackage{subcaption}
\usepackage{caption}
\usepackage{appendix}
\usepackage{multirow}
\usepackage{bm}

\newcommand{\blue}[1]{\textcolor{blue}{#1}}

\newcommand{\ARXIV}[1]{\href{http://arXiv.org/abs/#1}{\blue{arXiv:#1}}}

\def\proclaim#1{\par \smallskip\noindent {\bf #1}\bgroup\it\ }
\def\endproclaim{\egroup\par\smallskip}

\newtheorem{theorem}{Theorem}[section]

\newtheorem{assumption}{Assumption}[section]

\newtheorem{corollary}{Corollary}[section]

\newtheorem{remark}{Remark}[section]
\newtheorem{example}{Example}[section]

\newbox\TempBox \newbox\TempBoxA

\newcommand{\Ex}{\textsf{E}}

\newcommand {\Prob}{\textsf{P}}
\newcommand{\Proj}{\textsf{Pj}}
\newcommand{\Var}{\textup{Var}}

\renewcommand{\theequation}{\thesection.\arabic{equation}}

\def\Var{\textsf{Var}} 

\renewcommand{\theequation}{\thesection.\arabic{equation}}
\def\underwiggle 1{
\ifmmode\setbox\TempBox=\hbox{$ 1$}\else\setbox\TempBox=\hbox{
1}\fi \setbox\TempBoxA=\hbox to \wd\TempBox{\hss\char'176\hss}
\rlap{\copy\TempBox}\smash{\lower9pt\hbox{\copy\TempBoxA}} }

\begin{document}

\thispagestyle{empty}

\begin{center}
 { \LARGE\bf  Covariate-Adaptive
Randomization in Clinical Trials without Inflated Variances}  \ARXIV{2602.10760}
\end{center}

\begin{center} {\sc
\href{https://mypage.zjgsu.edu.cn/tjjy/zlx2_en/main.htm}{\blue{Li-Xin Zhang}}\footnote{Research supported by National Key
R\&D Program of China (No. 2024YFA1013502),  NSF of China (Grant Nos. U23A2064) and   the Summit Advancement Disciplines of Zhejiang Province (Zhejiang Gongshang University - Statistics).
}
}\\
{\sl \small School  of Statistics and Mathematics, Zhejiang Gongshang University, Hangzhou 310018} \\
(Email:stazlx@mail.zjgsu.edu.cn)    \\
\end{center}

\begin{abstract} 
 Covariate adaptive randomization (CAR)  procedures are extensively used to reduce the likelihood of covariate imbalances occurring in clinical trials. 
 In   literatures, a lot of    CAR   procedures have been proposed so that the specified covariates are balanced well between treatments. However, the variance of the imbalance of the unspecified covariates may be inflated comparing to the one under the simple randomization. The variance inflation causes the usual test of   treatment effects being not valid and adjusting the test being   not an easy work. In this paper, we propose a new kind of covariate adaptive randomization   procedures to balance covariates between two treatments  with a ratio  $\rho:(1-\rho)$.  Under this kind of CAR procedures, the convergence rate of the imbalance of the specified covariates is $o(n^{1/2})$, and at the same time the asymptotic variance of the imbalance of any unspecified (observed or unobserved) covariates does not exceed the one under the simple randomization. The "shift problem” found by Liu, Hu, and Ma  (2025)  will not appear under the new CAR procedures.  


\end{abstract}

\section{Introduction}
\label{s:intro}
\setcounter{equation}{0}

Covariate adaptive randomization (CAR) procedures have often been implemented for
balancing treatment allocation over influential covariates in clinical trials \citep{McEntegart2003, Rosenberger2015, Taves2010, Lin2015}. The main purposes are to enhance the comparability of 
treatment groups and to increase the efficiency of the statistical inference.
Many CAR procedures have been proposed to balance treatment assignments within strata and over covariate margins, provided that the covariates under consideration are discrete  (categorical with two or more levels) \citep{Pocock1975, Antognini2011, Hu2012, Rosenberger2015}. The theories on CAR randomization procedures are  established progressively in the past decade\citep{Hu2012,Ma2015,HuZhang2020,Hu2022}. For  balancing continuous covariates, besides the method of discretization, a variety of methods  have been proposed by minimization of prescribed imbalance measures
\citep{Hoehler1987,Nishi2004, Endo2006, Stigsby2010,Ciolino2011,Lin2012,Ma2013,Qin2016,Jiang2017,Li2019}.  Recently,  
\cite{Ma2024} unified a lot of CAR procedures by  proposing a general CAR procedure for balancing   specified covariate features $\bm \phi(\bm X_i)$ between two treatments $1$ and $2$, where   $\bm\phi(\bm X_i)=(\phi_1(\bm X_i),\ldots,\phi_q(\bm X_i))$ can contain additional covariate features besides the covariate vector $\bm X_i=(X_{i,1},\ldots,X_{i,p})$ of a unit in clinical trals, such as quadratic and interaction terms.  It is shown that the convergence rate of the imbalance vectors $\sum_{i=1}^{n}(T_i-1/2)\bm \phi(\bm X_i)$ is $o_P(n^{\epsilon_0})$, $\epsilon_0<1/2$, where $n$ is the sample size, $\bm X_i$ is the covariate vector of the $i$-th unit, and $T_i=1$, $0$ stands for the $i$-unit being allocated to treatment $1$ and $2$, respectively.   Under some suitable conditions on the density  of $\bm\phi(\bm X_i)$, $\sum_{i=1}^{n}(T_i-1/2)\bm \phi(\bm X_i)$ can attain the fastest convergence rate $O_P(1)$,  and at the same time, for any unspecified covariate feature $m(\bm X_i)$,  the asymptotic normality of its normalized imbalance   $\sum_{i=1}^{n}(T_i-1/2)  m(\bm X_i)/\sqrt{n}$ holds and the asymptotic variance $\sigma_m^2$ is nonzero when $m(\bm X_i)$ is not a linear function of $\bm \phi(\bm X_i)$. However, $\sigma_m^2$ may exceed $\Ex[m^2(\bm X_i)]/4$ that is the variance under the complete randomization in which each unit is allocated to treatments $1$ and $2$ with equal probability $1/2$. This issue is referred to as the variance inflation. The variance inflation  under the procedure of \cite{Pocock1975}  is recently found by \cite{Zhao2022} by simulation. Because the asymptotic variance of the  test of   treatment effects is usually a function of $\sigma_m^2$ for some unspecified covariate feature $m(\bm X_i)$,
the variance inflation causes such test being not valid \citep{Ma2024,Zhang2023}.  The closed form of $\sigma_m^2$ is also unknown so that   estimating it and adjusting the test are not easy  works. A lot of methods for the adjustment to obtain robust statistical inferences in the CAR clinical trials have being  proposed\citep{Bugin2018,YeShao2020,LiuYang2020,MaTuLiu2022,Wang2023}.   The purpose of this paper is to propose a new kind of CAR procedures so that the specified covariate features $\bm \phi(\bm X_i)$ can be balanced well and at the same time the asymptotic variance $\sigma_m^2$ of the normalized imbalance of any unspecified covariate feature $m(\bm X_i)$ is not inflated and has a closed form. 

From the proofs \citep{HuZhang2020, Ma2024, Zhang2023}, we can find that the convergence of $\sum_{i=1}^{n}(T_i-1/2) m(\bm X_i)/\sqrt{n}$ and the finiteness of   its asymptotic variance $\sigma_m^2$  are strongly based on the theory of Markov chains and the property that 
the allocation probabilities are  symmetric about treatments so that,  as a Markov chain, $\sum_{i=1}^{n}(T_i-1/2)\bm \phi(\bm X_i)$ 
has a symmetric   invariant  distribution and thus $(T_i-1/2) m(\bm X_i)$s have zero means under the invariance  distribution for any $m(\bm X_i)$. 
Recently, \cite{Liu2025}  generalized \cite{Ma2024}'s procedure to  balance covariances between two treatments with a ratio  $\rho:(1-\rho)$ instead of half to half. The convergence rate of the  imbalance vectors $\sum_{i=1}^{n}(T_i-\rho)\bm \phi(\bm X_i)$ of the  specified covariate feature $\bm\phi(\bm X_i)$ is also shown to be  $o_P(n^{\epsilon_0})$, $\epsilon_0<1/2$. However, by their theory, for an unspecified covariate feature $m(\bm X)$, $\sum_{i=1}^{n}(T_i-\rho) m(\bm X_i)/n$ may converge to a nonzero constant $c_m$ due to the fact that $\sum_{i=1}^{n}(T_i-\rho)\bm \phi(\bm X_i)$ is no longer  a symmetric Markov chain when $\rho\ne 1/2$. This  issue is refer to as
"shift problem" of covariate $m(\bm X_i)$.  The shift problem is a bad news because, in  testing hypotheses  
in the covaraite-adaptive randomized clinical trials, $\sum_{i=1}^{n}(T_i-\rho) m(\bm X_i)=o_P(n)$ is a basic condition \citep{Ma2015,Ma2020,Ma2024,Zhu2025,Liu2023,Liu2025}, and sometimes the convergence rate $O_P(\sqrt{n})$ is needed \citep{MaTuLiu2022}.  Basing on data in clinical trials under a randomization procedure which has the shift problem,  the estimator may be not consistent and the statistical inference may lead to a wrong conclusion.  We will show that our proposed CAR procedures are valid for all the case of $0<\rho<1$. Under the proposed CAR procedure, the "shift problem" will no longer appear  and the asymptotic variance of $\sum_{i=1}^{n}(T_i-\rho) m(\bm X_i)/\sqrt{n}$ will be always controlled by $\rho(1-\rho)\Ex[m^2(\bm X_i)]$, the variance under the simple randomization in which each unit is randomized to treatments $1$ and $2$ respectively with probabilities $\rho$ and $1-\rho$,  for any $m(\bm X_i)$. 
 
 In the next section, the framework of the new CAR procedures are proposed with theories on its properties. In Section \ref{sec:application}, as an application, the properties of a simple test of treatment effects under the proposed CAR procedure are studied based on the observed responses on treatments. It is shown that the test always controls the type I error rate without any information of the covariates, and the test can be adjusted to have the precise  type I error rate and increase the power when the data of 
 the covariate features  $\bm\phi(\bm X_i)$s   in the randomization stage are also available in the  analysis stage. The proofs are given in the last section.

\section{  New family of CAR procedures}
\label{s:procedure}
\setcounter{equation}{0}

\subsection{Framework}\label{sec:Procedure_Framework}

Suppose that $n$ experimental units are to be assigned to two treatment groups.
Let $T_i$ be the assignment of the $i$-th unit, i.e., $T_i=1$ for treatment 1 and $T_i=0$
for treatment 0.
Consider $p$ covariates available for each unit and use $\boldsymbol{X}_i=(X_{i,1}, ...,X_{i,p}) \in \mathbb{R}^p$ to denote the covariates of the $i$-th unit. Besides  $\bm X_i$, the unit may have other random features such as responses and other covaraites which are observable or unobservable and not considered in the randomization, we denote them by $Z_i,W_i$.
We assume that   $\{(\bm{X}_i, Z_i,W_i); i=1,2,\ldots,n\}$ are independent and identically distributed as a random vector $(\bm{X}, W, Z)$.
 We consider a procedure to balance general covariate features $\bm\phi(\bm X_i)$, defined by a feature map $\bm\phi(\bm x): \mathbb R^p\to \mathbb R^q$ that maps $\bm X_i$ into a $q$-dimensional feature space. Here, $q$ is usually larger than $p$ so
that $\bm \phi(\bm X_i)$   has more features than the original covariates. 
The  purpose of the covariate-adaptive randomization is to balance $\bm\phi(\bm X_i)$s between treatments $1$ and $2$ with a ratio $\rho:(1-\rho)$, where $0<\rho<1$.
The vector $\bm \Lambda_n=\sum_{i=1}^n (T_i-\rho)\bm \phi(\bm X_i)$ is the imbalance measure of the covariate feature $\bm \phi(\bm X)$. 
The numerical imbalance measure $Imb_n$ is defined 
  as the squared Euclidean norm of the imbalance vector $\bm \Lambda_n$, i.e.,
  $$Imb_n=\Big\|\sum_{i=1}^n (T_i-\rho)\bm \phi(\bm X_i)\Big\|^2. $$
  In this paper, we denote $\langle \bm x, \bm y\rangle=\bm x\bm y^{\prime}=\sum_{t=1}^dx_iy_j$, $\|\bm x\|=\sqrt{\langle \bm x, \bm x\rangle}$, $\bm x^{\otimes 2}=\bm x^{\prime}\bm x=\big(x_ix_j\big)_{d\times d}$, for real vectors $\bm x=(x_1, \ldots, x_d)$ and $\bm y=(y_1, \ldots, y_d)$.  
 
 Notice that
\begin{align*}
 Imb_{n}-Imb_{n-1}= &\|\bm \Lambda_{n}\|^2-\|\bm \Lambda_{n-1}\|^2\\
 = & 2(T_n-\rho)\langle\bm\Lambda_{n-1},\phi(\bm X_{n})\rangle
+(T_n-\rho)^2\|\bm \phi(\bm X_{n})\|^2. 
\end{align*}
When the $n$-th unit is ready to be randomized to treatments, the values of $\bm\Lambda_{n-1}$ and $\bm X_{n}$ are available and given. The decreasing or increasing of the   numerical imbalance measure $Imb_{n-1}$ depends on the value of $\langle\bm\Lambda_{n-1},\bm\phi(\bm X_{n})\rangle$ and the allocation probability. We define the allocation probability as a decreasing function of $\langle\bm\Lambda_{n-1},\bm\phi(\bm X_n)\rangle$ so that when the value of $|\langle\bm\Lambda_{n-1},\bm\phi(\bm X_{n})\rangle|$ is large, the value of $Imb_n$ will be drifted back.  
The new $\bm\phi$-based CAR procedure  is defined as follows:
\begin{enumerate}
	
	\item[(1)]
	Randomly assign the first unit    to  treatment $1$ with the probability  $\rho$ and   to treatment $2$ with the probability $1-\rho$.
	
	\item[(2)]
	 Suppose that $(n-1)$ units have been assigned to a treatment $(n > 1)$ and the $n$-th unit
is to be assigned, and the results of assignments $T_1,\ldots, T_{n-1}$ of previous stages and all covariates $\bm X_1,\ldots, \bm X_n$ up  to the $n$-th unit are observed.  Calculate the  imbalance vector $\bm\Lambda_{n-1}$.
	
	\item[(3)]Assign the $n$-th unit to the treatment with the probability
 	\begin{align}\label{eq:allocation-2}
	\ell_n=:\Prob(T_n=1|\bm{X}_n,...,\bm{X}_{1},T_{n-1},...,T_{1}) =\ell \left( \frac{\langle\bm \Lambda_{n-1},\bm \phi(\bm X_n)\rangle}{(n-1)^{\gamma}}\right),
	\end{align}
	where $0<\gamma<1$, $\ell(x):\mathbb R\to (0,1)$  is a non-increasing function with
\begin{align}
  \ell(0)&=\rho, \;
  \ell^{\prime}(0)<0 \text{ and } \ell(x) \text{ is twice differentiable at } x=0. \label{eq:allocationfunction1}
\end{align}
	\item[(4)]
	Repeat the last two steps until all units are assigned.
\end{enumerate}
The function $\ell(x)$ in \eqref{eq:allocation-2} is called the allocation function.  A proposed  choice of $\ell$ is 
\begin{equation}\label{eq:normallocation2}
\ell (x)= \frac{\big[2\rho\Phi(-x)\big]\wedge 1+ 1-\big[2(1-\rho)\Phi(x)\big]\wedge 1}{2}.
\end{equation}
 where  $\Phi$ is the standard normal distribution, $x\wedge y=\min\{x,y\}$ and $x\vee y=\max\{x,y\}$. Another simple choice of $\ell$ is
\begin{equation}\label{eq:normallocation1}
\ell (x)= \Phi(-x+u_{\rho}),
\end{equation}
where     $u_{\alpha}$ is the $\alpha$th quantile of a standard normal distribution   such that $\Phi(u_{\alpha})=\alpha$.  
 
 According to \eqref{eq:thm:Property1:3}, the  parameter $\gamma$ controls the size of the 
 imbalance measure $Imb_n$.  The smaller the parameter $\gamma$ ,  the smaller is $Imb_n$, but possibly the larger is the imbalance $\|\sum_{i=1}^n (T_i-\rho)Z_i\|^2$ of the unspecified feature $Z_i$  that is not a linear fuction of $\bm\phi(\bm X_i)$ due to \eqref{eq:thm:Property1:9}. The value of  $\gamma$ is suggested to be in $[0.5,1)$ so that the asymptotic expectation of $\|\sum_{i=1}^n (T_i-\rho)Z_i\|^2/n$ does not exceed the one under the simple randomization for any $Z_i$s.

\subsection{Theoretical Properties}


In this subsection we shall study the theoretical properties of the CAR  procedures with a feature map $\bm\phi(\bm{x})=(\phi_1(\bm{x}),\ldots,\phi_q(\bm{x})):\mathbb{R}^p\mapsto\mathbb{R}^q$.
Our goal is to obtain the convergence rate of covariate imbalance measured by
$\bm{\Lambda}_n=\sum_{i=1}^n(T_i-\rho)\bm\phi(\bm{X}_i)$ and the asymptotic normality of $\sum_{i=1}^n(T_i-\rho)Z_i$. 

To obtain the asymptotic properties, we require the following assumptions.

\begin{assumption}\label{asm:iid}
The   sequence $\{(\bm{X}_i,Z_i,W_i); i=1,2,\ldots\}$ is a sequence of independent and identically distributed as a random vector $(\bm{X}, W, Z)$.
\end{assumption}

\begin{assumption}\label{asm:moment}
  The feature map $\phi(\bm{X})$ satisfies that $\Ex[\|\phi(\bm{X})\|^{r_0}]$ is finite, $r_0> 2$.
\end{assumption}

  In the sequel of this paper,  for  sequences of real numbers $\{a_n\}$, $\{b_n\}$ $(b_n>0)$ and a sequence of random vectors $\{\bm Y_n\}$,   $a_n=O(b_n)$ means $|a_n/b_n|\le C$ for some $C>0$, $a_n=o(b_n)$ means $a_n/b_n\to 0$,  $\bm Y_n=O(b_n)$ (resp. $o(b_n)$)  in 
 $L_r$ means $(\Ex\|\bm Y_n\|^r)^{1/r}=O(b_n)$ (resp. $o(b_n)$), $\bm Y_n=o_P(b_n)$ means that $\|\bm Y_n\|/b_n\to 0$ in probability, and $\bm Y_n=O_P(b_n)$ means that $\|\bm Y_n\|/b_n$ is bounded in probability.   A positive constant which depends on $\alpha$ and $\beta$ is denoted by $C_{\alpha,\beta}$. 

\begin{theorem}\label{thm:Property1}
Suppose that Assumptions  \ref{asm:iid} and \ref{asm:moment} are satisfied.
\begin{description}
  \item[\rm (i) ] We have
  \begin{equation}\label{eq:thm:Property1:1}\left(\Ex\left[\|\bm \Lambda_n\|^r\right]\right)^{1/r} \le  C_{r,r_0}\Big(n^{\frac{\gamma}{2} } +n^{\frac{\gamma}{2}- \frac{(r_0+2-r)\gamma-2}{2r(r_0+1-r)}\big)}\Big), \; 2\le r\le r_0,\end{equation}
\begin{equation}\label{eq:thm:Property1:2}
\left(\Ex\left[\|\bm \Lambda_n\|^r\right]\right)^{1/r}  
\le    C_{r,r_0}\Big(n^{\frac{\gamma}{r} } +n^{ \frac{(r_0-2)\gamma +2}{2r(r_0+1-r)}}\Big), \;
 \; 1\le r \le 2.\end{equation}
In particular, 
\begin{equation}\label{eq:thm:Property1:3} \Ex \left[\|\bm \Lambda_n\|^2\right]=O\big(n^{\gamma}+n^{\frac{(r_0-2)\gamma +2}{2(r_0-1)}}\big) \text{ and } \Ex\left[\|\bm \Lambda_n\|^{r_0}\right] =O(n^{\frac{\gamma r_0}{2}+ 1-\gamma }) =o(n^{r_0/2}). 
\end{equation}
  \item[\rm (ii)] 
If   $\gamma> \frac{2}{3r_0-2}$ (more precisely,   $\gamma>\frac{4}{r_0^2+5}$ when $2<r_0\le 3$, $\gamma> \frac{2}{3r_0-2}$ when $3\le r_0\le 4$,  $\gamma>\frac{8}{(r_0+2)^2+4}$ when $r_0\ge 4$), then
\begin{equation}\label{eq:thm:Property1:7}
\Ex|\ell_n-\rho| \to 0
\end{equation} 
and 
\begin{equation}\label{eq:thm:Property1:8}
  \frac{\Ex\left|\sum_{i=1}^n (T_i-\rho)Z_i\right|}{n} \to  0 \text{ if } \Ex[|Z|]<\infty.
\end{equation} 
\item[\rm (iii)]  If $r_0\ge 4$,  $\gamma\ge 2/r_0$ and $\Ex Z^2<\infty$, then
  \begin{align}\label{eq:thm:Property1:9}
     \sum_{i=1}^n (T_i-\rho)Z_i=&\sum_{i=1}^n\big((T_i-\rho)\widetilde{Z}_i-\Ex[(\ell_i-\rho)\widetilde{Z}_i|\bm \Lambda_{i-1}]\big)\nonumber\\
     & +O(n^{\gamma/2}+n^{1-\gamma}) \text{ in } L_1,
\end{align} 
where   $\widetilde{Z}=Z-\langle\bm\phi(\bm X),\bm x_0\rangle$ and 
$\bm x_0=\arg\min\{ \Ex\big(Z-\langle\bm\phi(\bm X),\bm x\rangle\big)^2:\bm x\in \mathbb{ R}^q\}$.
\item[\rm (iv)] If $r_0\ge 4$,  $\gamma>\frac{1}{2}$, 
$\Ex Z^2<\infty$, $\Ex W^2<\infty$ and $\Ex[W]=0$, then
\begin{equation} \label{eq:thm:Property1:10}   \left(\frac{\sum_{i=1}^n (T_i-\rho)Z_i}{ \sqrt{n}}, \frac{\sum_{i=1}^n W_i}{ \sqrt{n}}\right)
 \overset{d}\to N(0,\vec{\sigma}_Z^2)\times N(0,\Ex W^2),
  \end{equation}
where $\vec{\sigma}_Z^2=\rho(1-\rho)\Ex[\widetilde{Z}^2]$, and
\begin{equation} \label{eq:thm:Property1:11} \frac{\Ex\Big(\sum_{i=1}^n(T_i-\rho)Z_i, \sum_{i=1}^nW_i\Big)^{\otimes 2}}{n}\to diag(\vec{\sigma}_Z^2,\Ex W^2),
  \end{equation} 
  when $(r_0-2)\gamma\ge 2$.
\end{description}
\end{theorem}

  \begin{remark}\label{rk:1}
 \begin{description}
 \item[\rm (i) Balance of the specified covariates:] By \eqref{eq:thm:Property1:3}, $\bm \Lambda_n=O_P(n^{\epsilon_0})$  for an $\epsilon_0\in [\gamma/2,1/2)$. If all moments of $\|\bm \phi(\bm X)\|$ are finite, then  $\bm \Lambda_n=O_P(n^{\gamma/2})$.
     \item[\rm (ii) No shift problem:]   It is easily seen that  \eqref{eq:thm:Property1:8} is equivalent to 
      \begin{equation}\label{eq:thm:Property1:8ad}\frac{\sum_{i=1}^n T_i Z_i}{\sum_{i=1}^n T_i}\overset{P}\to \Ex Z\;\text{ and } \;
         \frac{\sum_{i=1}^n (1-T_i) Z_i}{\sum_{i=1}^n (1-T_i)}\overset{P}\to \Ex Z  \;\text{ if } \;
         \Ex|Z|<\infty. 
      \end{equation}
      As the law of large numbers, \eqref{eq:thm:Property1:8ad} is essential for an estimator to be consistent. 
      A CAR procedure  is proposed \cite{Liu2025} for balancing covariate features $\bm \phi(\bm X_i)$ between treatments with an unequal ratio $\rho:(1-\rho)$. When the covariate  $m(\bm X_i)$ is not the linear function of  $\bm \phi(\bm X_i)$, there is no theoretical guarantee for  $\frac{1}{n}\sum_{i=1}^n (T_i-\rho)m(\bm X_i)  \overset{P}\to 0$,   it may happen that $\frac{1}{n}\sum_{i=1}^n (T_i-\rho)m(\bm X_i) \overset{P}\to c_m\ne 0$ and so \eqref{eq:thm:Property1:8ad} no longer holds. This issue is referred as "shift problem".    The simulation results presented in the Supplementary Material of \cite{Liu2025}  show that the shift problem on the additional covariates is common when balancing continuous covariates by using their CAR procedure. Now,   \eqref{eq:thm:Property1:8} shows that the shift problem no longer exists under our proposed CAR procedure.

  \item[\rm (iii) Closed form of variance:]  $\langle\bm\phi(\bm X),\bm x_0\rangle$ is the orthogonal projection of $Z$ on $\bm \phi(\bm X)$, we denote it by $\Proj[Z|\bm \phi(\bm X)]$. Thus, $\vec{\sigma}_{Z}^2=\Ex\big(Z-\Proj[Z|\bm \phi(\bm X)]\big)^2$.
      This closed form makes $\vec{\sigma}_{Z}^2$ can be estimated basing on the observations of $Z$ and $\bm\phi(\bm X)$.
      
       It is easily seen that 
$\Ex\big(Z-\Proj[Z|\bm \phi(\bm X)]\big)^2  =\Ex Z^2 -\Ex\big(\Proj[Z|\bm \phi(\bm X)]\big)^2\le \Ex Z^2$.
Thus,  $\vec{\sigma}_Z^2=0$ if and only if $Z\in Span\{\bm \phi(\bm X)\}=: \{\langle\bm x,\bm \phi(\bm X)\rangle:\bm \in \mathbb R^q\}$, and $\vec{\sigma}_Z^2=(1-\rho)\rho\Ex[Z^2]$ if and only if $\Ex[\bm\phi(\bm X)Z]=\bm 0$, i.e., $Z\perp\bm \phi(\bm X)$. As a result, $\vec{\sigma}_Z^2=\Var(Z-\Ex[Z|\bm\phi(\bm X)])+\vec{\sigma}_{\Ex[Z|\bm\phi(\bm X)]}^2$. Further, if $1\in Span\{\bm\phi(\bm X)\}$, then $\vec{\sigma}_Z^2=\vec{\sigma}_{Z-\Ex Z}\le \Var(Z)$. 

\item[\rm (iv) No variance inflation:] 
    The simple randomization is a randomization procedure which randomizes each unit to treatment $1$ with the same probability $\rho$. For the simple randomization, $\vec{\sigma}_Z^2=\rho(1-\rho)\Ex Z^2$.   So, the proposed CAR procedure always has smaller asymptotic variance than the  simple randomization for all observed or unobserved feature $Z_i$ with $\Proj[Z|\bm \phi(\bm X)]\ne 0$. The variance inflation  is no longer a problem under our proposed CAR procedure. 
     \item[\rm (v) Independence of exogenous variables:]   
         The normalized imbalance $\sum_{i=1}^n (T_i-\rho)Z_i/\sqrt{n}$  of $Z_i$s is an endogenous variable  of the design, and the normalized sum $\sum_{i=1}^nW_i/\sqrt{n}$ of $W_i$s can be regraded as an  exogenous variable. \eqref{eq:thm:Property1:10} means that they are  asymptotically independent   regardless $W_i$ and $Z_i$ themselves are independent or not.
   \end{description}
\end{remark}

\begin{remark}\label{rk:2} If $\ell(x)$ is thrice differentiable at $x=0$, and $\ell^{\prime\prime}(0)=0$ or $\Ex[(\bm\phi(\bm X))^{\prime}\bm \phi(\bm X)\widetilde{Z}]=\bm 0$, then  \eqref{eq:thm:Property1:10} remains true when $r_0\ge 7$ and $1/3<\gamma<1$. 

The allocation function $\ell$ as chosen in   \eqref{eq:normallocation1} with $\rho=1/2$  or \eqref{eq:normallocation2}
satisfies 
$\ell^{\prime}(0)=-\frac{1}{\sqrt{2\pi}}$ and $\ell^{\prime\prime}(0)=0$.

If the distribution of $(\bm \phi(\bm X), Z)$ is symmetric, i.e. $-(\bm \phi(\bm X), Z)$ and $ (\bm \phi(\bm X), Z)$ have the same distribution, then $\Ex[(\bm\phi(\bm X))^{\prime}\bm \phi(\bm X)\widetilde{Z}]=\bm 0$.

If $\Ex[Z|\bm \phi(\bm X)]\in Span\{\bm \phi(\bm X)\}$, then $ \Ex[\widetilde{Z}|\bm \phi(\bm X)]=0$ and thus $\Ex[(\bm\phi(\bm X))^{\prime}\bm \phi(\bm X)\widetilde{Z}]=\bm 0$.
\end{remark}

In Theorem \ref{thm:Property1} and Remark \ref{rk:2}, for the asymptotic normality \eqref{eq:thm:Property1:10},   $\gamma >1/2$ or  $\gamma>1/3$ is required,  respectively.   Next, we propose an allocation function such that \eqref{eq:thm:Property1:10} may  hold for all $0<\gamma<1$. The function is defined by 
\begin{equation}\label{eq:normallocation3} \ell(x)=\ell_{propose}(x)=\Phi(-|x|+u_{\rho/2})\vee (\rho-\lambda x)\wedge \Phi(|x|+u_{(\rho+1)/2}), 
\end{equation}
 where $\lambda>0$ is a tuning parameter.  In \eqref{eq:normallocation3}, $\Phi(-|x|+u_{\rho/2})$, $\Phi(|x|+u_{(\rho+1)/2})$ can be respectively replaced by $\alpha $, $\beta$ with $0<\alpha<\rho<\beta<1$. 
 
 \begin{theorem}\label{thm:Property2}
Suppose  Assumption  \ref{asm:iid},  and that \ref{asm:moment} is satisfied with $r_0\ge 4$,
$\Ex Z^2<\infty$, $\Ex W^2<\infty$ and $\Ex W=0$. Let $\ell(x)$ be defined as in \eqref{eq:normallocation3}. If $(r_0+2)\gamma>4$, then
\begin{equation}\label{eq:thm:Property2:1}
 \big(\Ex\|\bm{\Lambda}_n\|^r\big)^{1/r}=O(n^{\gamma /2}) \; \text{  for   } r=r_0/2,
 \end{equation}
 \eqref{eq:thm:Property1:10} and \eqref{eq:thm:Property1:11}   hold.
\end{theorem}
 
\begin{corollary}\label{thm:Property3}
Suppose  Assumption  \ref{asm:iid},  and that \ref{asm:moment} is satisfied for all $r_0>2$.
\begin{description}
  \item[\rm (i)]  We have 
  \begin{equation}\label{eq:thm:Property3:1}
 \big(\Ex\|\bm{\Lambda}_n\|^r\big)^{1/r}=O(n^{\gamma /2}) \; \text{  for all } r>0.
 \end{equation}
 \item[\rm (ii)] Suppose
   $\Ex Z^2<\infty$, $\Ex W^2<\infty$, $\Ex W=0$, and let $\ell(x)$ be defined as in \eqref{eq:normallocation3}.
Then  
\begin{equation}\label{eq:thm:Property3:2}
     \sum_{i=1}^n (T_i-\rho)Z_i=\sum_{i=1}^n\big((T_i-1)\widetilde{Z}_i-\Ex[(\ell_i-\rho)\widetilde{Z}_i|\bm \Lambda_{i-1}]\big)+O(n^{\gamma/2}) \text{ in } L_r
\end{equation} 
for all $r>0$. Moreover,  for all the case of  $0<\gamma<1$,  \eqref{eq:thm:Property1:10} and \eqref{eq:thm:Property1:11}  hold.
\end{description}
\end{corollary}

From the proofs given in Section \ref{sec:proof}, the conclusions in Theorems \ref{thm:Property1},\ref{thm:Property2} and corollary \ref{thm:Property3} remain true when the factor  $(n-1)^{\gamma}$ in the definition \eqref{eq:allocation-2} of the allocation probability is replaced by a general factor $\lambda_{n-1}$ with 
$0<c_0\le \lambda_n/n^{\gamma}\le C_0<\infty$. 

\begin{example}\label{example:1}  (Discrete covariates).  Suppose $\bm X=(X_1,\ldots,X_p)$  consists of $p$ discrete covariates, with the
$l$-th covariate $X_l$ having $m_l$ levels, $x_l^j, j=1,\cdots, m_l$, $l = 1, \cdots , p$. 
We let  $\bm\phi=\bm\phi(\bm X)$ be the following feature map in our framework,
 $$ \bm\phi=\left(\sqrt{w_o},\underset{\sum m_l \text{marginal terms}}{\underbrace{\ldots,\sqrt{w_{m,l}}I\{X_l=x_l^{k_l^{\ast}}\},\ldots}},
 \underset{\prod m_l \text{within-stratum terms}}{\underbrace{\ldots,\sqrt{w_{s}}I\{X_1=x_1^{k_1^{\ast}},\ldots,X_p=x_p^{k_p^{\ast}}\}\ldots}}\right), $$
where $w_o$, $w_{m,l}$, $w_s$ are the corresponding nonnegative weights and at least one of them is no zero.

\end{example}

Let $D_n=\sum_{i=1}^n(T_i-\rho)$, $D_n(l,k_l^{\ast})=\sum_{i=1}^n(T_i-\rho)I\{X_{i,l}=x_l^{k_l^{\ast}}\}$, $D_n(k_1^{\ast},\ldots, k_p^{\ast})=\sum_{i=1}^n(T_i-\rho)I\{X_{i,1}=x_1^{k_1^{\ast}},\ldots,X_{i,p}=x_p^{k_p^{\ast}}\}$. Then
$$\|\bm \Lambda_n\|^2 
=w_0D_n^2+\sum_{l=1}^{p}w_{m,l}\sum_{k_l=1}^{m_l}D_n^2(l;k_l)+w_s\sum_{k_1,\ldots,k_p}D_n^2(k_1,\ldots,k_p). $$

When $\bm X_n=(x_1^{k_1^{\ast}},\ldots,x_p^{k_p^{\ast}})$,
\begin{align*} \Lambda_{n-1}(k_1^{\ast},\ldots,k_p^{\ast})=:&\langle\bm\Lambda_{n-1},\bm\phi(\bm X_n)\rangle\\
=&w_0D_{n-1}+\sum_{l=1}^pw_{m,l}D_{n-1}(l;k_l^{\ast})+w_sD_{n-1}(k_1^{\ast},\ldots, k_p^{\ast}).
\end{align*}
Then, the allocation probability of the $n$th-unit is 
\begin{align}\label{eq:example1:1}
   \Prob\Big(T_n=1 & \Big|\bm X_1,\ldots, \bm X_{n-1}, T_1,\ldots,T_{n-1},\bm X_n=(x_1^{k_1^{\ast}},\ldots,x_l^{k_p^{\ast}})\Big)
   \nonumber\\
& =\ell\left(\frac{\Lambda_{n-1}(k_1^{\ast},\ldots,k_p^{\ast})}{(n-1)^{\gamma}}\right). 
\end{align}
When $w_s\ne 0$, the procedure is related to the stratified randomization. When $w_s=0$, $w_ {m,l}\ne 0$, $l=1,\ldots,p$, the procedure is related to the \cite{Pocock1975}'s marginal procedure. When $w_s=w_ {m,l}= 0$, $l=1,\ldots, p$, the procedure is related to \cite{Efron1971}'s biased coin design. 

By Corollary \ref{thm:Property3}, 
\begin{equation}\label{eq:example1:2} \|\bm \Lambda_n\|^2=O(n^\gamma) \text{ in }  L_r  \text{ for all } r>0. 
\end{equation}
 Thus, when $w_s\ne 0$,
$D_n(k_1,\ldots,k_p)=O_P(n^{\gamma/2})$ and so, for any $m(\bm X)$, $\sum_{i=1}^n (T_i-\rho)m(\bm X_i)=O_P(n^{\gamma/2})$.  
When $w_s=0$, we will have
\begin{align*}
D_n(l,k_l)=   O_P(n^{\gamma/2}) \; \text{ and } &\sum_{i=1}^n(T_i-\rho)f(X_{i,l})= O_P(n^{\gamma/2}), \text{ if } w_{m,l}\ne 0,\\
D_n=& O_P(n^{\gamma/2}).
\end{align*}
In any case, we have
\begin{equation}\label{eq:example1:3}\frac{\Ex(\sum_{i=1}^n (T_i-\rho)m(\bm X_i))^2}{n}\to \vec{\sigma}_m^2 \text{ and } \frac{\sum_{i=1}^n (T_i-\rho)m(\bm X_i)}{\sqrt{n}}
\overset{d}\to N(0,\vec{\sigma}_m^2) 
\end{equation}
with $\vec{\sigma}_m^2=\vec{\sigma}_{m-\Ex m}^2\le \rho(1-\rho)\Var(m(\bm X))$
since $1\in Span(\bm \phi(\bm X))$. In particular,
\begin{equation}\label{eq:example1:4} \frac{D_n(k_1,\ldots,k_p)}{\sqrt{n}}\overset{d}\to N(0,\vec{\sigma}^2(k_1,\ldots,k_p))  
\end{equation}
with $ \vec{\sigma}^2(k_1,\ldots,k_p)\le \rho(1-\rho) p(k_1,\ldots,k_p)(1-p(k_1,\ldots,k_p))$, where $p(k_1,\ldots,k_p)=\Prob(X_1=x_1^{k_1},\ldots,X_p=x_p^{k_p})$. 

When $\rho=1/2$, the allocation probability is similar to that of \cite{Zhang2021}, where \eqref{eq:example1:2} is obtained, but 
the properties \eqref{eq:example1:3} and \eqref{eq:example1:4} are not found. 

When $\rho=1/2$ and $\gamma=0$, our CAR procedure is similar to that of \cite{Hu2012}. The theory of this kind of CAR procedures can be found in \cite{HuZhang2020} and \cite{Zhang2023}.  It is shown that \eqref{eq:example1:3} still holds, but $\vec{\sigma}_m^2$ has no closed from and may exceed $\Ex[m^2(\bm X)]/4$ unless $\Ex[m(\bm X)|\bm\phi(\bm X)]\in Span\{\bm \phi(\bm X)\}$. In particular, when $w_s=w_0=0$, the procedure is \cite{Pocock1975}'s marginal procedure.  \cite{Zhao2022}  demonstrated via examples that along some directions, there is an increased variance of within-stratum imbalances under \cite{Pocock1975}'s  procedure compared to the simple randomization.

\begin{example}\label{example:2}   When a linear model is adopted to fit the responses of treatments over a covariate feature  $\bm X=(X_1,\ldots X_p)$. Adopting the feature $\bm\phi=(1,X_1,\ldots,X_p)$ in our framework will yields
$$ \sum_{i=1}^n (T_i-\rho)=O(n^{\gamma}) \text{ and } \sum_{i=1}^n (T_i-\rho)X_{i,t}=O(n^{\gamma}) \text{ in } L_r,$$
when $X_1,\ldots, X_p$ have finite moments of each order. At the same time, for any nonlinear function $m=m(X_1,\ldots,X_p)$ with finite variance, then
the asymptotic variance of $\sum_{i=1}^n (T_i-\rho)m(\bm X_i)/\sqrt{n}$ is $\rho(1-\rho)\Var(m-\Proj[m|\bm X])$ which is controlled by $\rho(1-\rho)\Var(m)$. The procedure will not cause the "shift problem" and the variance inflation. 
\end{example}
  
The results in Theorems \ref{thm:Property1}, \ref{thm:Property2} and Corollary \ref{thm:Property3} does not include the case of $\gamma=0$. Though  $\gamma=0$ is not our proposed choice, we give the results about this case as follows.
\begin{theorem}\label{thm:Property4} Let $\gamma=0$. 
Suppose that Assumptions  \ref{asm:iid} and \ref{asm:moment} are satisfied. We have the following conclusions.
\begin{description}
  \item[\rm (i)]  \eqref{eq:thm:Property1:1} and \eqref{eq:thm:Property1:2} hold with $\gamma=0$, i.e. $\left(\Ex\left[\|\bm \Lambda_n\|^r\right]\right)^{1/r} \le  C_{r,r_0} n^{ \frac{1}{r(r_0+1-r)}}$, $ r\in [1,  r_0]$. In particular,  $\Ex[\|\bm{\Lambda}_n\|^2]=O\big(n^{\frac{1}{r_0-1}}\big)$, and $=O\big(n^{\frac{8}{(r_0+1)^2}}\big)$ when $r_0\ge 3$.
  \item[\rm (ii)] If $\Ex|Z|<\infty$ and $\Ex[Z|\bm \phi(\bm X)]\in Span\{\bm \phi(\bm X)\}$, then \eqref{eq:thm:Property1:8} holds and
      \begin{equation}\label{eq:thm:Property4:2}
      \Ex\left(\sum_{i=1}^n(T_i-\rho) Z_i\right)^2\le 2n\Ex Z^2 +C_0  \Ex Z^2\Ex\|\bm \Lambda_n\|^2 \text{ if } \Ex Z^2<\infty.
      \end{equation}
  \item[\rm (iii)] If $\Ex Z^2<\infty$,    $\Ex[Z|\bm \phi(\bm X)]\in Span\{\bm \phi(\bm X)\}$, and $\big(Z-\Ex[Z|\bm \phi(\bm X)]\big)^2\in Span\{\bm \phi(\bm X)\}$ when $\rho\ne 1/2$, then 
      \begin{equation} \label{eq:thm:Prperty4:3}  \frac{ \sum_{i=1}^n(T_i-\rho)Z_i }{\sqrt{n}}\overset{d}\to N(0,\vec{\sigma}_Z^2),
  \end{equation} 
   where $\vec{\sigma}_Z^2=(1-\rho)\rho\Ex[(Z-\Proj[Z|\bm \phi(\bm X)])^2]$.  Further, if $\Ex W^2<\infty$, $\Ex[W]=0$ and also $
   \big(Z-\Ex[Z|\bm \phi(\bm X)]\big)W$ is  in $Span\{\bm \phi(\bm X)\}$, then we have \eqref{eq:thm:Property1:10}.
\end{description}
\end{theorem}
When $\gamma=0$, the CAR procedure  is similar to \cite{Liu2025}'s procedure.  It does not cause the ``shift problem" when $\Ex[Z|\bm \phi(\bm X)]\in Span\{\bm \phi(\bm X)\}$, and the variance is not inflated when $\rho= 1/2$.   When $\bm X$ is discrete as in Example \ref{example:1} and $\bm\phi(\bm X)$ given there with $w_s\ne 0$, i.e., the stratified randomization is adopted,   then $\Ex[Z|\bm \phi(\bm X)]$, $\big(Z-\Ex[Z|\bm \phi(\bm X)]\big)^2$, $\big(Z-\Ex[Z|\bm \phi(\bm X)]\big)W$ are all in $Span\{\bm \phi(\bm X)\}$, and so the asymptotic normality   \eqref{eq:thm:Property1:10}  holds with a variance $\vec{\sigma}_Z^2$ being not inflated. This conclusion is consistent with Theorem 3 of \cite{Liu2025}. 

\section{An application}
\label{sec:application}
As an application,  we    study hypothesis testing of treatment effects under the proposed 
covariate-adaptive procedure.
   For each unit $i$, denote $Y_i(1)$ and $Y_i(2)$ as the potential outcomes under the treatments $1$ and $2$, respectively, with $\Ex[Y_i(t)]=\mu_t$ and $\Var(Y_i(t))=\sigma_t^2$, $t=1,2$.  The observed outcome is
$
Y_i = T_i Y_i(1) + ( 1 - T_i ) Y_i(2).
$
The treatment effect is defined as $\tau =\Ex \{ Y_i(1)\}-\Ex \{ Y_i(2)\}$.
Also, denote $\bm X_i=(X_{i,1},\ldots, X_{i,p})$   as the covariate of unit $i$ as before.  Assume that $(  Y_i(1), Y_i(2), \bm X_i)$, $i=1,\ldots, n$, are  independent and identically distributed as $(Y(1), Y(2),\bm X)$.  We suppose that the units are randomized sequentially by  the proposed CAR procedure with feature $\bm\phi(\bm X)$.
We consider the test of
\begin{equation}\label{htest} H_0:\tau=:\Ex \{ Y_i(1)\}-\Ex \{ Y_i(2)\} = 0\;\;  \text{versus} \;\; H_A : \tau\ne  0.
\end{equation}
Let $\overline{Y}_{n,1}=\frac{1}{N_{n,1}}\sum_{i=1}^n T_iY_i$, $\overline{Y}_{n,2}=\frac{1}{N_{n,2}}\sum_{i=1}^n(1-T_i)Y_i$,
$\widehat{\sigma}_{n,1}^2=\frac{1}{N_{n,1}-1}\sum_{i=1}^n T_i(Y_i-\overline{Y}_{n,1})^2$ and
$\widehat{\sigma}_{n,2}^2=\frac{1}{N_{n,2}-1}\sum_{i=1}^n(1-T_i)(Y_i-\overline{Y}_{n,2})^2$,  where $N_{n,1}=\sum_{i=1}^n T_i$, $N_{n,2}=n-N_{n,1}$, and $\frac{0}{0}$ is defined to be $0$. Then  $\overline{Y}_{n,1}$, $\overline{Y}_{n,2}$, $\widehat{\sigma}_{n,1}^2$ and $\widehat{\sigma}_{n,2}^2$ are  estimators of $\mu_1$, $\mu_2$, $\sigma_1^2$ and $\sigma_2^2$, respectively. We consider the classical test statistic 
\begin{equation}\label{TS}
  \mathcal{T}(n)=\frac{\overline{Y}_{n,1}-\overline{Y}_{n,2}}{ \sqrt{\widehat{\sigma}_{n,1}^2/N_{n,1}+\widehat{\sigma}_{n,2}^2/N_{n,2}}}.  
\end{equation}

\begin{theorem} \label{thm:test} Suppose that the proposed CAR procedure with a feature $\bm\phi(\bm X)$ and an allocation function $\ell(x)$ as in \eqref{eq:normallocation3} is applied to balance covariates between treatments $1$ and $2$  with a ratio  $\rho:(1-\rho)$, and  Assumptions \ref{asm:iid}, \ref{asm:moment} for all $r_0>2$  are satisfied.
Then \begin{equation}\label{eq:const} \overline{Y}_{n,t}\overset{P}\to \mu_t, \;\; \widehat{\sigma}_{n,t}^2\overset{P}\to \sigma_t^2, \;\;t=1,2, \end{equation}
\begin{equation}\label{eq:esttau}
 \sqrt{n}\big(\overline{Y}_{n,1}-\overline{Y}_{n,2}-\tau\big)\overset{D}\to N(0,\sigma_{\tau}^2),
  \end{equation}
where 
\begin{equation}\label{eq:asyvaroftest} \sigma_{\tau}^2=\frac{\sigma_1^2}{\rho}+\frac{\sigma_2^2}{1-\rho}-\rho(1-\rho)\sigma_{Pj}^2
=\frac{\breve{\sigma}_1^2}{\rho}+\frac{\breve{\sigma}_2^2}{1-\rho}+\breve{\sigma}_3^2,
\end{equation} 
\begin{equation}
\begin{aligned} \label{eq:asyvariancesdef}
 \sigma_{Pj}^2=&\Ex\Big(\Proj\big[ \frac{Y(1)-\mu_1}{\rho} + \frac{Y(2)-\mu_2}{1-\rho}\big|\bm \phi(\bm X)\big]\Big)^2,\\
 \breve{\sigma}_t^2=& \Ex\Big(Y(t)-\mu_t-\Proj\big[ Y(t)-\mu_t\big|\bm \phi(\bm X)\big]\Big)^2,\;\;t=1,2,\\
 \breve{\sigma}_3^2= &\Ex\Big(\Proj\big[  Y(1)-\mu_1-(Y(2)-\mu_2)\big|\bm \phi(\bm X)\big]\Big)^2. 
\end{aligned}
\end{equation}
Moreover, 
 \begin{itemize}
\item[(i)]  Under $H_0:\tau=0$,
\begin{equation}\label{eq:test1}
\mathcal{T}(n)\overset{D}{\to}N(0,\sigma_{\mathcal T}^2),\; \text{with}, \;\sigma_{\mathcal T}^2=\frac{\sigma_{\tau}^2}{\sigma_e^2}\le 1,
\end{equation}
 where
 $\sigma_e^2= \frac{\sigma_1^2}{\rho}+\frac{\sigma_2^2}{1-\rho}$.
 \item[(ii)] under $H_A:\tau\neq 0$, consider a sequence of local alternatives, i.e., $\tau=\delta/\sqrt{n}$ for a fixed $\delta \neq 0$ (we suppose the distributions of $Y(t)-\mu_t$, $t=1,2$, do not depend on $n$),  then
\begin{equation}\label{eq:test2}
\mathcal{T}(n)\overset{D}{\to}N(\Delta,\sigma_{\mathcal T}^2),\; \text{with}\; \Delta=\frac{\delta}{\sigma_e}.
\end{equation}
\end{itemize}
\end{theorem}
For the simple randomization, $\sigma_{\tau}=\sigma_e$ and then $\sigma_{\mathcal{T}}=1$.

From \eqref{eq:test1}, it follows that if we  reject $H_0$ when $ |\mathcal{T}(n)| \ge u_{1-\alpha/2}$, then
$$ \lim_{n\to \infty}\Prob( |\mathcal{T}(n)| \ge u_{1-\alpha/2 }| H_0)\le \alpha. $$
   The test always controls the type I error rate but is conservative when $\sigma_{\mathcal{T}}<1$ ($\Proj\big[ (Y(1)-\mu_1)/\rho + (Y(2)-\mu_2)/(1-\rho)\big|\bm \phi(\bm X)\big]\not\equiv 0$).  Under the local alternatives, the asymptotic power is
\begin{align*}
 \lim_{n\to \infty}\Prob( |\mathcal{T}(n)| \ge u_{1-\alpha/2 }| H_A:\tau=\frac{\delta}{\sqrt{n}})= &  \Prob\left(\left|\frac{|\delta|}{\sigma_{\tau}}+N(0,1)\right|\ge u_{1-\alpha/2}\frac{\sigma_e}{\sigma_{\tau}}\right)\\
\le & \Prob\left(\left|\frac{|\delta|}{\sigma_{\tau}}+N(0,1)\right|\ge u_{1-\alpha/2}\right).
\end{align*}
When $|\delta|\le  \sigma_e u_{1-\alpha/2} $, the power is an increasing function of $\sigma_{\tau}^2\in(0,\sigma_e^2]$, and so the proposed CAR procedure losses the power comparing to the simple randomization.   When $|\delta|>(1+\epsilon_0)\sigma_e u_{1-\alpha/2} $ which means that the signal is significantly stronger than the noise, the power is a decreasing function of $\sigma_{\tau}^2\in(0,\sigma_e^2]$,   and so the proposed CAR procedure increases the power comparing to the simple randomization, where $\epsilon_0$ is the solution of $\ln(1+2/\epsilon_0)=2(1+\epsilon_0)u_{1-\alpha/2}^2 $, and $\epsilon_0=9.153\times 10^{-4}$ when $\alpha=0.05$, $\epsilon_0=8.566\times 10^{-3}$ when $\alpha=0.10$.
   
  If   $\widehat{\sigma}_{n,\tau}^2$ is an consistent estimator     of ${\sigma}_{\tau}^2$, $t=1,2,3$, and we define the  adjusted test statistic as 
\begin{equation}\label{TSadj}
\mathcal{T}_{adj}(n)=\frac{\sqrt{n}(\overline{Y}_{n,1}-\overline{Y}_{n,2})}{ \widehat{\sigma}_{n,\tau} }.
\end{equation}
then 
$$ \lim_{n\to \infty}\Prob( |\mathcal{T}_{adj}(n)| \ge u_{1-\alpha/2 }| H_0)= \alpha, $$
$$ \lim_{n\to \infty}\Prob( |\mathcal{T}_{adj}(n)| \ge u_{1-\alpha/2 }| H_A:\tau=\frac{\delta}{\sqrt{n}})
= \Prob\left(\left|\frac{|\delta|}{\sigma_{\tau}}+N(0,1)\right|\ge u_{1-\alpha/2}\right). $$
The adjusted test restores the type I error rate and increases the power.  The power is greater than the one under the simple randomization when $ \frac{Y(1)-\mu_1}{\rho} + \frac{Y(2)-\mu_2}{1-\rho} \not\perp\bm \phi(\bm X)$. 

When  $\bm \phi(\bm X_i)$, $i=1,\ldots,n$, are also available in the analysis stage, we can obtain the consistent estimator as follow. 
Let $T_{i,1}=T_i$, $T_{i,2}=1-T_i$.   We regress the residuals $T_{i,t}(Y_i-\overline{Y}_{n,t})$s over $T_{i,t}\bm\phi(\bm X_i)$s by fitting the linear model
$$ T_{i,t}(Y_i-\overline{Y}_{n,t}) =T_{i,t}\bm\phi(\bm X_i)\bm \alpha_t^{\prime}+\zeta_{i,t}, \; i=1,\ldots, n, $$
and obtain the estimator $\widehat{\bm \alpha}_{n,t}$ of $\bm \alpha_t$ and the new residuals $\widehat{\zeta}_{i,t}=T_i(t)(Y_i-\overline{Y}_{n,t})-T_i(t)\bm\phi(\bm X_i)\widehat{\bm \alpha}_{n,t}^{\prime}$, $i=1,\ldots, n$, $t=1,2$.  Define 
\begin{equation}\label{eq:RegEst}
\begin{aligned} & \widehat{\breve{\sigma}}_{n,t}^2=\frac{1}{\sum_{i=1}^n T_i(t)-q-1} \sum_{i=1}^n \big(\widehat{\zeta}_{i,t}\big)^2, \; t=1,2,\\
& \widehat{\breve{\sigma}}_{n,3}^2 =\frac{1}{n-2}\sum_{i=1}^n  \big(\bm \phi(\bm X_i)[\widehat{\bm \alpha}_{n,1}-\widehat{\bm \alpha}_{n,2}]^{\prime}\big)^2 \; \text{ and }\\
& \widehat{\sigma}_{n,\tau}^2=:\widehat{\sigma}_{n,\tau,reg}^2
=n\left(\frac{\widehat{\breve{\sigma}}_{n,1}^2}{N_{n,1}}
+\frac{\widehat{\breve{\sigma}}_{n,2}^2}{N_{n,2}}
+\frac{\widehat{\breve{\sigma}}_{n,3}^2}{n}\right).
\end{aligned}
\end{equation}

\begin{theorem} \label{thm:regest} Under the conditions in Theorem \ref{thm:test}, 
$$\widehat{\breve{\sigma}}_{n,t}^2\overset{P}\to  \breve{\sigma}_t^2,  \;t=1,2,3 \text{ and 
} \widehat{\sigma}_{n,\tau,reg}^2\overset{P}\to \sigma_{\tau}^2.$$
\end{theorem} 
  
 Theorems \ref{thm:test} and \ref{thm:regest} will follow from \eqref{eq:thm:Property1:8} and \eqref{eq:thm:Property1:10}.  
  
 It is important that the asymptotic variance $\vec{\sigma}_Z^2$ in \eqref{eq:thm:Property1:10} has the closed form $\Ex(Z-\Proj[Z|\bm\phi(\bm X)])^2$ and does not exceed $\Ex Z^2$   because in this case the classical test is valid and the variance is estimable so that the test can be adjusted to have the precise type I error rate. 
  
  When the stratified randomization is adopted,   Theorems  \ref{thm:test} and \ref{thm:regest} remain true   even when $\gamma=0$ by Theorem
\ref{thm:Property4}.
 Under the stratified randomization, we always have $\vec{\sigma}_Z^2=\Ex(Z-\Proj[Z|\bm\phi(\bm X)])^2$. This is the main reason why the classical test under the stratified randomization is always valid and can be adjusted in varies ways.   

\bigskip
If only $Y_i, T_i, i=1,\ldots, n$ are available in the analysis stage, we can construct a moving block estimator of $\sigma_{\tau}^2$ as follows. Let $l=l_n$  be an integer between $1$ and $n$. Let  
$$ \overline{Y}_{1}(j,l)=\frac{\sum_{i=j+1}^{j+l}T_{i}Y_i}{\sum_{i=j+1}^{j+l}T_{i}},\;
\overline{Y}_{2}(j,l)=\frac{\sum_{i=j+1}^{j+l}(1-T_{i})Y_i}{\sum_{i=j+1}^{j+l}(1-T_{i})} \text{ and } \widehat{\tau}(j,l)=\overline{Y}_{1}(j,l)-\overline{Y}_{2}(j,l) $$
be respectively the estimators of $\mu_1$, $\mu_2$ and $\tau$ basing on the data $\{T_i,Y_i; i=j+1,\ldots,j+l\}$. Denote the sample variance of $\{\widehat{\tau}(j,l); j=0,1,\ldots, n-l\}$ by 
$$ SV_n= \frac{1}{n-l}\sum_{j=0}^{n-l}\Big(\widehat{\tau}(j,l)-\overline{\widehat{\tau}(\cdot,l)} \Big)^2, \text{ where } 
\overline{\widehat{\tau}(\cdot,l)}=\frac{1}{n-l+1}\sum_{j=0}^{n-l}\widehat{\tau}(j,l). $$
 Then $\widehat{\sigma}_{n,\tau,mb}^2=:l\cdot SV_n$ is a consistent estimator of $\sigma_{\tau}^2$. 
 
 \begin{theorem}\label{thm:mb} Suppose $l/n\to 0$ and $n^{\gamma/2+\epsilon_0}/l\to 0$ for an $\epsilon_0>0$. Assume $1\in Span\{\bm \phi(\bm X)\}$. 
 Under the conditions in Theorem \ref{thm:test},
 \begin{equation}\label{eq:thm:mb:1}
 \widehat{\sigma}_{n,\tau,mb}^2=l\cdot SV_n \overset{P}\to \sigma_{\tau}^2. 
 \end{equation}
 \end{theorem} 

\section{Proofs of the results} 
\label{sec:proof}

\subsection{The properties of the CAR procedures}
We write $\bm \phi_n=\bm\phi(\bm X_n)$, $\bm \phi=\bm \phi(\bm X)$.  
Let $\Sigma=\Ex\left[\bm \phi^{\prime}\bm \phi\right]$.  Without loss of generality, we assume that $\Sigma$ is non-singular. In fact, if $\Sigma=\bm 0$, then $\bm \phi=\bm \phi_n=\bm \Lambda_n=0$ a.s., and $\ell_n\equiv \rho$, \eqref{eq:thm:Property1:7}-\eqref{eq:thm:Property1:11}, \eqref{eq:thm:Property3:1}-\eqref{eq:thm:Property3:2} are obvious. If $\Sigma\ne \bm 0$, then there is an orthogonal matrix $U$ such that $\Sigma=U diag(\sigma_1^2,\ldots,\sigma_{q_{\ast}}^2,0,\ldots) U^{\prime}$, where $\sigma_t^2$s are positive numbers. Let ${\bm \psi}=\bm \phi U^{\prime}$ and ${\bm \psi}_i=\bm \phi_i U^{\prime}$. Then $\Ex[{\bm \psi}^{\prime}{\bm \psi}]=diag(\sigma_1^2,\ldots,\sigma_{q_{\ast}}^2,0,\ldots)$, $\bm\Lambda_n=\sum_{i=1}^n(T_i-\rho) {\bm \psi}_iU$,
$\|\bm \Lambda_n\|^2=\|\sum_{i=1}^n(T_i-\rho){\bm \psi}_i\|^2$, $\langle\bm \Lambda_n,\bm \phi_{n+1}\rangle
=\langle\sum_{i=1}^n(T_i-\rho){\bm \psi}_i,{\bm \psi}_{n+1}\rangle$. So, it is sufficient to consider ${\bm \psi}=({\psi}^{(1)},\ldots,{\psi}^{(q)})$. Since $\Prob({\psi}^{(t)}=0, t=q_{\ast}+1,\ldots q)=1$,   it is sufficient to consider the subverter ${\bm\psi}_{\ast} =({\psi}^{(1)},\ldots, {\psi}^{(q_{\ast})})$
with $\Ex[{\bm\psi}_{\ast}^{\prime}{\bm\psi}_{\ast}]=diag(\sigma_1^2,\ldots,\sigma_{q_{\ast}}^2)$ being non-singular.

Let  $\mathscr{F}_n=\sigma(\bm X_i,T_i,Z_i,W_i,i=1,2,\ldots,n)$ be the history sigma field generated by results of variables and allocations up to stage $n$.   Then, $
	\Prob(T_{n+1}=1|\mathscr{F}_n, \bm{X}_{n+1}) =\ell_{n+1}$ by \eqref{eq:allocation-2}, where 
 $\ell_{n}=\ell\left(\frac{\langle \bm \Lambda_{n-1},\bm \Phi_n\rangle}{(n-1)^{\gamma}}\right)$.  
 
 Let $r\in [2,r_0]$. 
Write   $\bm \psi_n=(T_n-\rho)\bm \phi(\bm X_n)$. Then $\bm \Lambda_{n+1}=\bm\Lambda_n+\bm\psi_n$. 
From the elementary inequality that
\begin{equation}\label{eq:element1}\|\bm u+\bm v\|^{r }-\|\bm u\|^{r }\le r  \langle\bm u,\bm v\rangle\|\bm u\|^{r  -2}+c_{r }   \left(\|\bm v\|^{r }+\|\bm v\|^2 \|\bm u\|^{r  -2}\right), \; r  \ge 2,
\end{equation}
where $c_{r }   $ is a constant which depends only on $r $,  we have
\begin{align*}
  &\Ex\left[\|\bm\Lambda_{n+1}\|^{r }|\mathscr F_n,\bm X_{n+1}\right]-\|\bm\Lambda_n\|^{r }\\
  \le &    \Ex\left[r  \langle\bm \Lambda_n,\bm \psi_{n+1}\rangle\|\bm   \Lambda_n\|^{r  -2}
  +c_{r }   \|\bm \psi_{n+1}\|^{r }
 + c_{r } \|\bm \psi_{n+1}\|^2\|\bm   \Lambda_n\|^{r  -2} \big|\mathscr F_n,\bm X_{n+1}\right]\\
 \le & \Ex\left[ r   (T_{n+1}-\rho)\langle\bm \Lambda_n,\bm \phi_{n+1})\rangle\|\bm   \Lambda_n\|^{r  -2}\big|\mathscr F_n,\bm X_{n+1}\right] 
   +c_{r }\left(\|\bm \phi_{n+1}\|^{r }
 +   \|\bm   \phi_{n+1}\|^2\|\bm   \Lambda_n\|^{r  -2}\right)\\
= & r    (\ell_{n+1}-\rho)\langle\bm \Lambda_n,\bm \phi_{n+1}\rangle\|\bm   \Lambda_n\|^{r  -2}
   +  c_{r }\left(\|\bm \phi_{n+1}\|^{r }+ \|\bm   \phi_{n+1}\|^2\|\bm   \Lambda_n\|^{r  -2}\right)\\
   = & -r  \left|\ell_{n+1}-\rho\right|\cdot\left|\langle\bm \Lambda_n,\bm \phi_{n+1}\rangle\right|\cdot\|\bm   \Lambda_n\|^{r  -2}  
   +  c_{r }\left(\|\bm \phi_{n+1}\|^{r }+ \|\bm   \phi_{n+1}\|^2\|\bm   \Lambda_n\|^{r  -2}\right),
\end{align*}
where the last equality is due to the fact that $\ell_{n+1}$ is a non-increasing function of 
$\langle\bm \Lambda_n,\bm \phi_{n+1}\rangle$.  Notice that there is a $\delta>0$, $|\ell(x)-\rho|\ge |x| |\ell^{\prime}(0)|/2$ when $|x|\le \delta$. While, when $|x|\ge \delta$, $ |\ell(x)-\rho|\ge (\rho-\ell(\delta))\wedge (\ell(-\delta)-\rho)>0$. It follows that, there exists a constant $\lambda_1>0$ such that $|\ell(x)-\rho|\ge \lambda_1 (|x|\wedge 1)$. Thus, $\Ex\left[\left|\ell_{n+1}-\rho\right|\cdot\left|\langle\bm \Lambda_n,\bm \phi_{n+1}\rangle\right|\big|\mathscr{F}_n\right]\ge \lambda_1 n^{\gamma}\hbar(\bm\Lambda_n/n^{\gamma} )$, where
$\hbar(\bm x)=\Ex\left[\left|\langle\bm x,\bm \phi\rangle\right|^2\wedge \left|\langle\bm x,\bm \phi\rangle\right|\right]$. 

Let $\lambda_{min}$ be the smallest eigenvalue of $\Sigma$. Then $\Ex\langle\bm x,\bm \phi\rangle^2=\bm x\Sigma\bm x^{\prime}\ge \lambda_{min}\|\bm x\|^2$. 
Choose $M\ge 1$ large enough such that $\Ex[\|\bm\phi\|^2I\{\|\bm\phi\|\ge M\}]\le \lambda_{min}/2$. Notice
\begin{align*}
\langle\bm x,\bm \phi\rangle^2\le & \langle\bm x,\bm \phi\rangle^2I\{|\langle\bm x,\bm \phi\rangle|\le 1\}\\
&+
M \|\bm x\|\, \big|\langle\bm x,\bm \phi\rangle\big|I\{|\langle\bm x,\bm \phi\rangle|> 1\}
+\|\bm x\|^2\|\bm \phi\|^2 I\{|\|\bm \phi\|| \ge M\}\\
\le & [1\vee (M\|\bm x\|)]\left[\langle\bm x,\bm \phi\rangle^2\wedge |\langle\bm x,\bm \phi\rangle|\right]+\|\bm x\|^2\|\bm \phi\|^2 I\{|\|\bm \phi\|| \ge M\}.
\end{align*}
It follows that
$$ \hbar(\bm x)\ge \frac{\Ex\langle\bm x,\bm \phi\rangle^2-\|\bm x\|^2\Ex[\|\bm\phi\|^2I\{\|\bm\phi\|\ge M\}]}{1\vee (M\|\bm x\|)} 
\ge \frac{\lambda_{min}\|\bm x\|^2/2}{1\vee (M\|\bm x\|)}. $$
Let $\lambda_2=\lambda_{min}/(2M)$. We have  $\hbar(\bm x)\ge  \lambda_2(\|\bm x\|^2\wedge \|\bm x\|)$. It follows that
 \begin{align}\label{eq:multi-drift-p-2}
  &\Ex\left[\|\bm\Lambda_{n+1}\|^{r }|\mathscr F_n\right]-\|\bm\Lambda_n\|^{r }\nonumber \\
   \le  & -r\lambda_1 n^{\gamma} \hbar\left( \frac{ \bm \Lambda_{n}}{n^{\gamma}}\right)\cdot\|\bm   \Lambda_n\|^{r  -2} 
  +c_{r}\left( \beta_{r}+\beta_2\|\bm   \Lambda_n\|^{r -2}\right),\nonumber \\
  \le &- r\lambda_1\lambda_2 \|\bm\Lambda_n\|^{r-1 }\left(  \frac{ \|\bm \Lambda_{n}\|}{n^{\gamma}}\wedge 1\right)
  +c_{r}\left( \beta_{r}+\beta_2\|\bm   \Lambda_n\|^{r -2}\right),
\end{align}
 where $\beta_{r}=\Ex[\|\bm \phi(\bm X)\|^{r}]$. Consider  the function
 $$ f(x)=-\alpha_1 x^{r-1}\left(\frac{x}{n^{\gamma}}\wedge 1\right)+\alpha_2x^{r-2}, \; x\ge 0, $$
 where $\alpha_1,\alpha_2>0$. When $ 0\le x\le n^{\gamma}$, 
\begin{align*}  f(x)= -\alpha_1\frac{x^{r}}{n^{\gamma}}+\alpha_2x^{r-2}
 \le & \begin{cases} 0, & \text{ when } \alpha_1\frac{x^2}{n^{\gamma}}\ge \alpha_2\\
 \alpha_2\Big(\frac{\alpha_2}{\alpha_1}\Big)^{\frac{r-2}{2}}n^{\frac{\gamma}{2}(r-2)},  
& \text{ when } \alpha_1\frac{x^2}{n^{\gamma}}< \alpha_2\end{cases} \\
 \le &\alpha_2\Big(\frac{\alpha_2}{\alpha_1}\Big)^{\frac{r-2}{2}}n^{\frac{\gamma}{2}(r-2)}. 
 \end{align*}
 Similarly, when $ x> n^{\gamma}$,
 $$ f(x)=-\alpha_1x^{r-1}+\alpha_2x^{r-2}\le \alpha_2\Big(\frac{\alpha_2}{\alpha_1}\Big)^{r-2}, \; x\ge 0. $$
 It follows that
 $$ f(x)\le \alpha_2\Big(\frac{\alpha_2}{\alpha_1}\Big)^{r-2}+\alpha_2\Big(\frac{\alpha_2}{\alpha_1}\Big)^{\frac{r-2}{2}}n^{\frac{\gamma}{2}(r-2)}.
 $$
 By choosing $\alpha_1=r\lambda_1\lambda_2/2$ and $\alpha_2=c_r\beta_2$,  it follows that
 \begin{align*} 
  &\Ex\left[\|\bm\Lambda_{n+1}\|^{r }|\mathscr F_n\right]-\|\bm\Lambda_n\|^{r }  
  \le  - \frac{r\lambda_1\lambda_2}{2} \|\bm\Lambda_n\|^{r-1 }\Big(  \frac{ \|\bm \Lambda_{n}\|}{n^{\gamma}}\wedge 1\Big)
  +C_r n^{\frac{\gamma}{2}(r-2)}.
\end{align*}
  
  Hence, we have the following recursive formula:
 \begin{align}
  &\Ex\left[\|\bm\Lambda_{n+1}\|^{r }|\right]-\Ex\left[\|\bm\Lambda_n\|^{r }\right]\nonumber \\
  \le & - \frac{r\lambda_1\lambda_2}{2} \Ex\Big[\|\bm\Lambda_n\|^{r-1 }\Big(  \frac{ \|\bm \Lambda_{n}\|}{n^{\gamma}}\wedge 1\Big)\Big]
  +C_{r}  n^{\frac{\gamma}{2}(r-2)}
  \label{eq:multi-drift-p-3}\\
  \le &C_{r}  n^{\frac{\gamma}{2}(r-2)}.
  \label{eq:multi-drift-p-4}
  \end{align}
   By \eqref{eq:multi-drift-p-4},
\begin{equation}\label{eq:proof:thm:9}  \Ex[\|\bm\Lambda_n\|^{r_0}]\le C_{r_0} \sum_{i=1}^{n-1}i^{\frac{\gamma }{2}(r_0-2)}\le C_{r_0} n^{\frac{\gamma }{2}(r_0-2)+1}, \;\; n\ge 1. 
\end{equation}

Let $E=\{\|\bm \Lambda_n\|\le n^{\gamma}\}$ and $Q_{n,r}=\Ex\big[\|\bm\Lambda_n\|^{r-1 }\big(  \frac{ \|\bm \Lambda_{n}\|}{n^{\gamma}}\wedge 1\big)\big]$. 
Then 
$$ \Ex\left[\|\bm \Lambda_n\|^rI_E\right]\le n^{\gamma}Q_{n,r}\;
\text{ and }\; \Ex\left[\|\bm \Lambda_n\|^{r-1}I_{E^c}\right]\le Q_{n,r}. $$
For $\frac{1}{f}+\frac{1}{g}=1$, $f,g>1$, by the H\"older inequality we have
\begin{align*}
&\Ex\left[\|\bm \Lambda_n\|^{r}I_{E^c}\right]=\Ex\left[\|\bm \Lambda_n\|^{\frac{r-1}{f}}\|\bm \Lambda_n\|^{\frac{r-1}{g}+1}I_{E^c}\right]\\
\le &\left(\Ex\left[\|\bm \Lambda_n\|^{r-1}I_{E^c}\right]\right)^{1/f}\left(\Ex\left[\|\bm \Lambda_n\|^{r-1+g}I_{E^c}\right]\right)^{1/g} 
\le Q_{n,r}^{1/f}\left(\Ex\left[\|\bm \Lambda_n\|^{r-1+g}\right]\right)^{1/g}. 
\end{align*}
It follows that
$$ \Ex[\|\bm \Lambda_n\|^r]\le n^{\gamma}Q_{n,r}+Q_{n,r}^{1/f}\left(\Ex\left[\|\bm \Lambda_n\|^{r-1+g}\right]\right)^{1/g}. $$

Now, let $m:=m_n\in\{0,1,\ldots, n\}$ be the last one for which
$$- \frac{r\lambda_1\lambda_2}{2} Q_{m,r}
  +C_rm^{\frac{\gamma}{2}(r-2)} \ge 0. $$
Then 
$$ Q_{m,r}\le \frac{2C_r}{r\lambda_1\lambda_2}m^{\frac{\gamma}{2}(r-2)}, $$
and, by \eqref{eq:multi-drift-p-3} and \eqref{eq:multi-drift-p-4},
\begin{align*} 
& \Ex\left[\|\bm \Lambda_n\|^{r}\right]\le \Ex\left[\|\bm \Lambda_{m+1}\|^{r}\right] \le \Ex\left[\|\bm \Lambda_{m}\|^{r}\right]
+C_{r}m^{\frac{\gamma}{2}(r-2)}\nonumber\\
\le &   m^{\gamma}Q_{m,r}+Q_{m,r}^{1/f}\left(\Ex\left[\|\bm \Lambda_m\|^{r-1+g}\right]\right)^{1/g}
+C_{r}m^{\frac{\gamma}{2}(r-2)}\nonumber\\
\le & C_r\left[m^{\frac{\gamma}{2} r} +\big(m^{\frac{\gamma}{2}(r-2)}\big)^{1/f}\left(\Ex[\|\bm \Lambda_m\|^{r-1+g}]\right)^{1/g}\right]. 
\end{align*}
It follows that 
\begin{equation}\label{eq:proof:thm:10}
  \Ex\left[\|\bm \Lambda_n\|^{r}\right]
  \le  C_r\Big[n^{\frac{\gamma}{2} r} +\big(n^{\frac{\gamma}{2}(r-2)}\big)^{1/f}\Big(\max_{m\le n} \Ex[\|\bm \Lambda_m\|^{r-1+g}]\Big)^{1/g}\Big], \; r\in [2,r_0]. 
\end{equation}
  
When $r\in[0,2]$, we have  the elementary inequalities that
$$ \|\bm u+\bm v\|^{r }-\|\bm u\|^{r }\le r  \langle\bm u,\bm v\rangle\|\bm u\|^{r  -2}+\frac{r}{2}\|\bm v\|^2\|\bm u\|^{r-2},
$$
and
$\|\bm u+\bm v\|^{r }-\|\bm u\|^{r }\le 2\|\bm v\|^r+\|\bm u\|^r.$
We use the first inequality when $\|\bm u\|\ge 1$ and the second inequality when $\|\bm u\|\le 1$. It follows that
\begin{align*}
 \|\bm u+\bm v\|^{r }-\|\bm u\|^{r }\le & r  \langle\bm u,\bm v\rangle\|\bm u\|^{r  -2}I\{\|\bm u\|\ge 1\}+\frac{r}{2}\|\bm v\|^2 +2\|\bm v\|^r+1\\
\le &r  \langle\bm u,\bm v\rangle\|\bm u\|^{r  -2}+ \|\bm v\|^2+2\|\bm v\|^r+2\|\bm v\|+1.
\end{align*}
Similar to \eqref{eq:multi-drift-p-2}, we have
 \begin{equation} \label{eq:multi-drift-p-2ad}
 \Ex\left[\|\bm\Lambda_{n+1}\|^{r }|\mathscr F_n\right]-\|\bm\Lambda_n\|^{r }\le - r\lambda_1\lambda_2 \|\bm\Lambda_n\|^{r-1 }\left(  \frac{ \|\bm \Lambda_{n}\|}{n^{\gamma}}\wedge 1\right)+C_r,\; r\in[1,2]. 
\end{equation}
Thus, 
$$
 \Ex\left[\|\bm\Lambda_{n+1}\|^{r }|\right]-\Ex\left[\|\bm\Lambda_n\|^{r }\right] 
  \le   - r\lambda_1\lambda_2  \Ex\Big[\|\bm\Lambda_n\|^{r-1 }\Big(  \frac{ \|\bm \Lambda_{n}\|}{n^{\gamma}}\wedge 1\Big)\Big]
  +C_{r},  \; r\in[1,2]. 
$$
   With this inequality taking the place of \eqref{eq:multi-drift-p-3}, we can show that
\begin{equation}\label{eq:proof:thm:10ad}
  \Ex\left[\|\bm \Lambda_n\|^{r}\right]
  \le  C_r\Big[n^{ \gamma } + \Big(\max_{m\le n} \Ex[\|\bm \Lambda_m\|^{r-1+g}]\Big)^{1/g}\Big], \; r\in [1,2]. 
\end{equation}

{\bf Proof of Theorem \ref{thm:Property1}}. 
(i) Let $r-1+g=r_0$. Notice \eqref{eq:proof:thm:9}. We obtain  
\begin{align*}
\left(\Ex\left[\|\bm \Lambda_n\|^r\right]\right)^{1/r} \le & C_{r,r_0}\left(n^{\frac{\gamma}{2} r} +n^{\frac{\gamma(r-2)}{2 f}+\frac{(r_0-2)\gamma +2}{2g}}\right)^{1/r}\\
\le &  C_{r,r_0}\left(n^{\frac{\gamma}{2} } +n^{\frac{\gamma}{2}-  \frac{(r_0+2-r)\gamma-2}{2r(r_0+1-r)}}\right), \; r\in[2,r_0],
\end{align*}
by \eqref{eq:proof:thm:10}, and 
\begin{align*}
\left(\Ex\left[\|\bm \Lambda_n\|^r\right]\right)^{1/r} \le & C_{r,r_0}\left(n^{\gamma} +n^{ \frac{(r_0-2)\gamma +2}{2g}}\right)^{1/r} 
\le    C_{r,r_0}\left(n^{\frac{\gamma}{r} } +n^{ \frac{(r_0-2)\gamma +2}{2r(r_0+1-r)}}\right), \; r\in[1,2],
\end{align*}
by \eqref{eq:proof:thm:10ad}.
 Thus, \eqref{eq:thm:Property1:1} and \eqref{eq:thm:Property1:2} are proven. Taking $r=2$ and $r=r_0$ in \eqref{eq:thm:Property1:1} respectively  yields \eqref{eq:thm:Property1:3}.

(ii) If $\gamma> \frac{2}{3r_0-2}$, then 
$$ \Ex\left[\|\bm \Lambda_n\|\right]\le \Big(\Ex \|\bm \Lambda_n\|^2\Big)^{1/2}=O\left(n^{\gamma/2}+n^{  \frac{(r_0-2)\gamma +2}{4(r_0-1)}}\right)=o(n^{\gamma}), $$
by \eqref{eq:thm:Property1:3}. If 
$\gamma>\frac{4}{r_0^2+5}$  and  $2<r_0\le 3$, then by taking $r=(r_0+1)/2$ in \eqref{eq:thm:Property1:2}, we have
$$ \Ex\left[\|\bm \Lambda_n\|\right]\le \Big(\Ex \|\bm \Lambda_n\|^r\Big)^{1/r}
=O\Big(n^{\frac{\gamma}{r}}+n^{\gamma+\frac{4-\gamma (r_0^2+5)}{(r_0+1)^2}}\Big)=o(n^{\gamma}). $$ 
If $\gamma>\frac{8}{(r_0+2)^2+4}$  and  $r_0\ge 4$, then by taking $r=r_0/2$ in \eqref{eq:thm:Property1:1}, we have
$$ \Ex\left[\|\bm \Lambda_n\|\right]\le \Big(\Ex \|\bm \Lambda_n\|^r\Big)^{1/r}
=O\Big(n^{\frac{\gamma}{2}}+n^{\gamma+\frac{2-\gamma ((r_0+2)^2+4)/4}{r_0(r_0/2+1)}}\Big)=o(n^{\gamma}). $$ 
Notice that $|\ell (x)-\rho|\le c|x|$ and $|\ell (x)-\rho|\le 1$. Thus,
\begin{align*}
  \Ex|\ell_n-\rho|
 \le &  c \Ex\left[\frac{\|\bm \Lambda_{n-1}\|\cdot \|\bm \phi_n\|}{(n-1)^{\gamma}}\right] 
 =   c \frac{\Ex\left[\|\bm \Lambda_{n-1}\|\right]\Ex[\|\bm \phi\|]}{(n-1)^{\gamma}} 
 \to   0.
 \end{align*}
 \eqref{eq:thm:Property1:7} is proven. Then
 \begin{equation} \label{eq:proof:th1:ell} \Ex|(\ell_i-\rho)Z_i|\le M\Ex|\ell_i-\rho|+\Ex[|Z|I\{|Z|\ge M\}] \to 0
 \end{equation}
  as  $i\to \infty$   and then $M\to \infty$.
 So, \eqref{eq:thm:Property1:8} holds.

(iii) For \eqref{eq:thm:Property1:9}, we denote $\Delta M_{i,\widetilde{Z}}=:(T_i-\rho)\widetilde{Z}_i-\Ex[(\ell_i-\rho)\widetilde{Z}_i|\bm \Lambda_{i-1}]$. Suppose $r_0\ge 4$ and $\Ex [Z^2]<\infty$. Then
\begin{align}\label{eq:proof:thm:expansion0}
&\sum_{i=1}^n(T_i-\rho)Z_i= \sum_{i=1}^n (T_i-\rho)\widetilde{Z}_i+\langle \bm \Lambda_n, \bm x_0\rangle\nonumber \\
=& \sum_{i=1}^n \Delta M_{i,\widetilde{Z}}+\langle \bm \Lambda_n, \bm x_0\rangle+\sum_{i=2}^{n-1}\Ex[(\ell_{i+1}-\rho)\widetilde{Z}_{i+1}|\bm \Lambda_{i}]. 
\end{align}
 We notice that
 $$ |\ell(x)-\rho-\ell^{\prime}(0)x|\le cx^2. $$
 Write $h_{\widetilde{Z}}(\bm x)=\left\langle \bm x,\Ex[\bm \phi \widetilde{Z}]\right\rangle$. Then $h_{\widetilde{Z}}(\bm x)\equiv 0$ since
 $\Ex[\bm \phi \widetilde{Z}]\equiv \bm 0$ due to the fact that $\langle\bm \phi,\bm x_0\rangle$ is the orthogonal projection of $Z$ on $\bm \phi$.  Thus,
 \begin{align}\label{eq:expansion}
  & \left|\Ex\left[(\ell_{n+1}-\rho)\widetilde{Z}_{n+1}\big|\bm \Lambda_{n}\right]\right|\nonumber\\
   =&  \left|\Ex\left[(\ell_{n+1}-\rho)\widetilde{Z}_{n+1}\big|\bm \Lambda_{n}\right]-\ell^{\prime}(0) h_{\widetilde{Z}}\Big(\frac{\bm \Lambda_n}{n^{\gamma}}\Big)\right|  
  \le    C  \Ex\left[ \frac{\|\bm \Lambda_n\|^2 \|\bm\phi\|^2|\widetilde{Z}|}{n^{2\gamma}}\big|\bm\Lambda_n\right]   \nonumber \\
  \le & 
   C\frac{\|\bm \Lambda_n\|^2}{n^{2\gamma}}\Big(\Ex \widetilde{Z}^{2}\Big)^{1/2}\Big(\Ex \|\bm\phi\|^{4}\Big)^{1/2}  
  \le  C_0\big(\Ex Z^2\big)^{1/2} \frac{\|\bm \Lambda_n\|^2}{n^{2\gamma}}.
 \end{align}
 Also, since $\bm x_0\Sigma\bm x_0^{\prime}=\Ex(\Proj[Z|\bm \phi])^2\le \Ex Z^2$, we have $\|\bm x_0\|\le \lambda_{min}^{-1/2}(\Ex Z^2)^{1/2}$. 
 Hence, we have
 \begin{align}\label{eq:martingale:approx}
 \sum_{i=1}^n(T_i-\rho)Z_i= \sum_{i=1}^n \Delta M_{i,\widetilde{Z}}+\theta_n(\Ex Z^2)^{1/2}\left(\| \bm \Lambda_n\| +
\sum_{i=2}^{n-1}\frac{\|\bm \Lambda_i\|^2}{i^{2\gamma}}\right), \;\; |\theta_n|\le C_0,
\end{align}
where $C_0$  does not depend on $Z$.

From \eqref{eq:thm:Property1:3}, it follows that  
\begin{align*}
\sum_{i=1}^n(T_i-\rho)Z_i=& \sum_{i=1}^n \Delta M_{i,\widetilde{Z}}+O\Big(n^{\frac{\gamma}{2}}
+n^{\frac{(r_0-2)\gamma +2}{4(r_0-1)}}+n^{1-\gamma}+n^{1+\frac{(r_0-2)\gamma +2}{2(r_0-1)}-2\gamma}\Big) \\
=&  \sum_{i=1}^n \Delta M_{i,\widetilde{Z}}+O\Big(n^{\frac{\gamma}{2}}
+n^{1-\gamma}\Big) \text{ in } L_1 
\end{align*}
by the fact  that   $ \frac{(r_0-2)\gamma +2}{2(r_0-1)}\le  \gamma$ when $\gamma\ge 2/r_0$. \eqref{eq:thm:Property1:9} is proven.
 
(iv) For \eqref{eq:thm:Property1:10}, notice that $\{(\Delta  M_{i,\tilde{Z}}, W_i), \mathscr{F}_i; i=1,2,\ldots\}$  is sequence of martingale differences. We have
  $|\Delta  M_{i,\tilde{Z}}|\le |\widetilde{Z}_i|+\Ex|\widetilde{Z}_i|$,
\begin{align*}
&\frac{1}{n}\sum_{i=1}^n \Ex\left[(\Delta  M_{i,\tilde{Z}})^2I\{\big|(\Delta  M_{i,\tilde{Z}})^2\ge \epsilon n\}\right]\\
\le &\frac{1}{n}\sum_{i=1}^n \Ex\left[(|\widetilde{Z}_i|+\Ex|\widetilde{Z}|)^2I\{\big|(|\widetilde{Z}_i|+\Ex|\widetilde{Z}|)^2\ge \epsilon n\}\right]\\
=& \Ex\left[(|\widetilde{Z}|+\Ex|\widetilde{Z}|)^2I\{\big|(|\widetilde{Z}|+\Ex|\widetilde{Z}|)^2\ge \epsilon n\}\right]\to 0,
\end{align*}
$$
 \frac{1}{n}\sum_{i=1}^n \Ex\left[W_i^2I\{ W_{i}^2\ge \epsilon n\}\right]
= \Ex\left[W^2I\{ W^2\ge \epsilon n\}\right]\to 0,
$$
i.e., the Linderberg condition is satisfied. Further, 
 $\Ex[W_i^2|\mathscr{F}_{i-1}]=\Ex W^2$,
\begin{align*}
\Ex\left[(\Delta M_{i,\tilde{Z}})^2\big|\mathscr{F}_{i-1}\right]=&
\Ex[((1-2\rho)\ell_i+\rho^2)\tilde{Z}_i^2|\bm \Lambda_{i-1}]-\big(\Ex[(\ell_i-\rho)\tilde{Z}_i|\bm \Lambda_{i-1}]\big)^2\\
\to & (1-\rho)\rho\Ex[\tilde{Z}_i^2] \text{ in } L_1, \\
\Ex\left[\Delta M_{i,\tilde{Z}}W_i\big|\mathscr{F}_{i-1}\right]=&
\Ex[(\ell_i-\rho)\tilde{Z}_iW_i|\bm \Lambda_{i-1}]
\to  0  \text{ in } L_1,
\end{align*}
by  applying \eqref{eq:proof:th1:ell} to $\tilde{Z}_i^2$ and $\tilde{Z}_iW_i$.

It follows that
\begin{equation}
\label{eq:clt:martingale}
\begin{aligned} &\frac{\Ex\big( M_{n,\widetilde{Z}}, \sum_{i=1}^n W_i\big)^{\otimes 2}}{n}\to diag(\vec{\sigma}^2, \Var(W)) \text{ and }  \\ &\left(\frac{M_{n,\widetilde{Z}}}{\sqrt{n}},\frac{\sum_{i=1}^n W_i}{\sqrt{n}}\right)\overset{d}\to N(0,\vec{\sigma}^2)\times N(0,\Var(W)), 
\end{aligned}
\end{equation}
by the central limit theorem for martingales\cite[Corollary 3.1]{Hall1980}. 
Hence, \eqref{eq:thm:Property1:10} holds by \eqref{eq:thm:Property1:9} and the fact that $1/2<\gamma<1$. 

When $(r_0-2)\gamma\ge 2$, applying \eqref{eq:thm:Property1:1} with $r=4$ yields that
$$\|\bm\Lambda_n\|=O\big(n^{\gamma/2}\big) \text{ in } L_4. $$
Thus, \eqref{eq:thm:Property1:9} holds in $L_2$ by \eqref{eq:martingale:approx}. Then, \eqref{eq:thm:Property1:11} holds by \eqref{eq:clt:martingale}.
$\Box$

\bigskip
{\bf Proof of Remark \ref{rk:2}}. Notice 
$ \big|\ell(x)-\rho -\ell^{\prime}(0)x-\frac{1}{2}\ell^{\prime\prime}(0)x^2\big|\le c|x|^3. $
Similar to \eqref{eq:expansion}, we have
\begin{align*} 
  &\left|\Ex\left[(\ell_{n+1}-\rho)\widetilde{Z}_{n+1}\big|\bm \Lambda_{n}\right]-\ell^{\prime}(0) h_{\widetilde{Z}}\Big(\frac{\bm \Lambda_n}{n^{\gamma}}\Big)-\frac{1}{2}\ell^{\prime\prime}(0)\frac{\bm \Lambda_n}{n^{\gamma}}\Ex\left[\bm\phi^{\prime}\bm\phi \widetilde{Z}\right]\frac{\bm \Lambda_n^{\prime}}{n^{\gamma}}\right|\nonumber \\
  \le & C  \Ex\left[ \frac{\|\bm \Lambda_n\|^3 \|\bm\phi\|^3| \widetilde{Z}|}{n^{2\gamma}}\big|\bm\Lambda_n\right]  
   \le  C    \frac{\|\bm \Lambda_n\|^3 }{n^{2\gamma}}\big( \Ex\|\bm\phi\|^6\big)^{1/2} \big(\Ex\widetilde{Z}^2\big)^{1/2} 
   \le   C_0(\Ex Z^2)^{1/2}\frac{\|\bm \Lambda_n\|^3}{n^{3\gamma}}.
 \end{align*}
Notice that $h_{\widetilde{Z}}\big(\bm x\big)\equiv \bm 0$ and $\ell^{\prime\prime}(0)\Ex[\bm\phi^{\prime}\bm\phi \widetilde{Z}]=\bm 0$.  It follows that
 \begin{align*}  
   \sum_{i=1}^n   & (T_i-   \rho)Z_i=   \sum_{i=1}^n \Delta M_{i,\widetilde{Z}} 
  +\theta_n (\Ex Z^2)^{1/2}\left(\|\bm \Lambda_n\|+
\sum_{i=2}^{n-1}\frac{\|\bm \Lambda_i\|^3}{i^{3\gamma}}\right) \;  (|\theta_n|\le C_0) \\
=&  \sum_{i=1}^n \Delta M_{i,\widetilde{Z}}+O\Big(n^{\gamma/2}+ n^{ \frac{\gamma}{2}-\frac{(r_0-1)\gamma-2}{2(r_0-2)}}+n^{1-3\gamma/2}
+n^{1-\frac{3\gamma}{2}-\frac{(r_0-1)\gamma-2}{6(r_0-2)}}\Big) \text{ in } L_1\\
=&  \sum_{i=1}^n \Delta M_{i,\widetilde{Z}}+o(n^{1/2}) \text{ in } L_1,
\end{align*}
by \eqref{eq:thm:Property1:1} (with $r=3$) and the facts that $\gamma/2<1$, $1-\frac{3\gamma}{2}<1/2$ and $ (r_0-1)\gamma-2\ge 0$ when $1/3<\gamma<1$ and $r_0\ge 7$,  where $C_0$ does not depend on $Z$. $\Box$. 

\bigskip

{\bf Proof of Theorem \ref{thm:Property2}}. \eqref{eq:thm:Property2:1} follows from \eqref{eq:thm:Property1:1} by letting $r=r_0/2$ and noticing $(r_0+2-r)\gamma-2\ge 0$. 

 Notice $\Phi(u_{\rho/2})<\rho<\Phi(u_{(\rho+1)/2})$. Let $K\ge 2$. By the continuity, there exists $\delta>0$ such that$$\ell(x)=\rho-\lambda x, \;\; |x|\le \delta. $$
When $|x|\ge \delta$, we have $|\ell(x)-\rho|\le 1\le \delta^{-K}|x|^K$, $|x|\le \delta^{-K+1}|x|^K$. So, there exists a constant $C_K$ such that
$$|\ell(x)-\rho+\lambda x|\le c_K|x|^K. $$
Similar to \eqref{eq:expansion}, we have
\begin{align*}  
  &\left|\Ex\left[(\ell_{n+1}-\rho)\widetilde{Z}_{n+1}\big|\bm \Lambda_{n}\right]  \right|=\left|\Ex\left[(\ell_{n+1}-\rho)\widetilde{Z}_{n+1}\big|\bm \Lambda_{n}\right]+\lambda h_{\widetilde{Z}}\Big(\frac{\bm \Lambda_n}{n^{\gamma}}\Big) \right|\nonumber \\
 & \;\; \le  c_K  \Ex\left[ \frac{\|\bm \Lambda_n\|^K \|\bm\phi\|^K|\widetilde{Z}|}{n^{2\gamma}}\big|\bm\Lambda_n\right]   \le  c_K(\Ex Z^2)^{1/2}(\Ex \|\bm \phi\|^{2K})^{1/2}\frac{\|\bm \Lambda_n\|^K}{n^{K\gamma}}.
 \end{align*}
 It follows that
 \begin{align}   \label{eq:proof:thm:2:2}
   &\sum_{i=1}^n    (T_i-   \rho)Z_i -M_{n,\widetilde{Z}}\nonumber \\
  =&\theta_n (\Ex Z^2)^{1/2}\left(\|\bm \Lambda_n\|+ (\Ex \|\bm \phi\|^{2K})^{1/2}
\sum_{i=2}^{n-1}\frac{\|\bm \Lambda_i\|^K}{i^{K\gamma}}\right) \;  (|\theta_n|\le \lambda_{min}^{-1/2}\vee c_K).
\end{align}
 Now, let $K=r_0/2$. By \eqref{eq:thm:Property1:3}, $\|\bm\Lambda_n\|=o(n^{1/2})$ in $L_2$, and
 $$ \frac{\|\bm \Lambda_n\|^K}{n^{K\gamma}}=O\left(\frac{n^{r_0\gamma/4+(1-\gamma)/2}}{n^{r_0\gamma/2}}\right) 
 = O\left( n^{\frac{1}{2} -\frac{(r_0+2)\gamma}{4}}\right)=o(n^{-1/2}) \text{ in } L_2. $$
 Thus, by \eqref{eq:proof:thm:2:2}, 
 $$ \sum_{i=1}^n (T_i-\rho)Z_i=M_{n,\widetilde{Z}}+o(n^{1/2}) \text{ in } L_2, 
 $$ 
 which, together with \eqref{eq:clt:martingale}, implies  \eqref{eq:thm:Property1:10} and \eqref{eq:thm:Property1:11}. $\Box$
 
\bigskip
{\bf Proof of Corollary \ref{thm:Property3}}.   \eqref{eq:thm:Property3:1}, \eqref{eq:thm:Property1:10} and \eqref{eq:thm:Property1:11} hold by Theorem \ref{thm:Property2}. 
 For \eqref{eq:thm:Property3:2}, we notice \eqref{eq:proof:thm:2:2}. 
 Choose $K>2$ such that $K\gamma/2>1$. By \eqref{eq:thm:Property3:1},
 $$\sum_{i=1}^{n-1}\frac{\|\bm \Lambda_i\|^K}{i^{K\gamma}}\le C\sum_{i=1}^{\infty} i^{-K\gamma/2}<\infty\;\; \text{ in } L_r. $$
 The proof of \eqref{eq:thm:Property3:2} is completed by \eqref{eq:proof:thm:2:2} and \eqref{eq:thm:Property3:1}. 
 $\Box$

\bigskip
{\bf Proof of Theorem \ref{thm:Property4}.} For (i), notice that \eqref{eq:multi-drift-p-2} and \eqref{eq:multi-drift-p-2ad} still hold for $\gamma=0$. Thus, \eqref{eq:thm:Property1:1} and \eqref{eq:thm:Property1:2} still hold for $\gamma=0$.
Letting $r=2$ and $g=r_0-1$ in \eqref{eq:thm:Property1:1} yields $\Ex\left[\|\bm \Lambda_n\|^{2}\right]=O(n^{\frac{1}{r_0-1}})$.  When $r_0\ge 3$. Let $r=(1+r_0)/2$  in \eqref{eq:thm:Property1:1}. Then $r\ge 2$, $g=r$. Thus, $\Ex\left[\|\bm \Lambda_n\|^{r}\right]=O(n^{1/r})$, and so $\Ex\left[\|\bm \Lambda_n\|^{2}\right]=O(n^{2/r^2})=O(n^{\frac{8}{(r_0+1)^2}})$. The proof of (i) is completed.  
  
  For (ii), suppose $\Ex[Z|\bm \phi]=\langle\bm \phi, \bm z_0\rangle\in Span\{\bm \phi\}$. 
Let $\widetilde{Z}_i=Z_i-\Ex[Z_i|\bm \phi_i]$. Then 
$$ \sum_{i=1}^n (T_i-\rho)Z_i=\sum_{i=1}^n(T_i-\rho)\widetilde{Z}_i + \langle\bm \Lambda_n, \bm z_0\rangle, $$
and
$\{(T_i-\rho)\widetilde{Z}_i,\mathscr{F}_i; i=1,2,\ldots\}$ is a sequence of martingale differences with $|(T_i-\rho)
\widetilde{Z}_i|\le |Z_i|+\|\bm z_0\|\cdot \|\bm \phi_i\|$. Thus, the martingale differences are uniformly integrable which implies 
$$ \frac{\sum_{i=1}^n(T_i-\rho)\widetilde{Z}_i}{n}\to 0 \text{ in } L_1  $$
 \citep[Theorem 2.22]{Hall1980}. \eqref{eq:thm:Property1:8} is proved by (i). 

Suppose $\Ex Z^2<\infty$. In this case $\langle \bm \phi, \bm z_0\rangle =\Proj[Z|\bm \phi]$ and thus  $|\langle \bm \Lambda_n, \bm z_0\rangle|\le C_0(\Ex Z^2)^{1/2}\|\bm \Lambda_n\|$ as we have shown. For the martingale $\sum_{i=1}^n(T_i-\rho)\widetilde{Z}_i$,
we have
$$ \Ex\big(\sum_{i=1}^n(T_i-\rho)\widetilde{Z}_i\big)^2\le n\Ex \widetilde{Z}^2\le n\Ex Z^2. $$
(ii) is proven.

For (iii), by the central limit theorem\cite[Theorem 3.2]{Hall1980}, it is sufficient to show that
\begin{equation}\label{eq:proof:thm4:1}\frac{1}{n}\sum_{i=1}^n (T_i-\rho)^2\widetilde{Z}_i^2  \to \rho(1-\rho)\Ex\widetilde{Z}^2 \;\; \text{ in } L_1. 
\end{equation} 
Notice
$$  (T_i-\rho)^2\widetilde{Z}_i^2 =\rho(1-\rho)\Ex\widetilde{Z}^2+\rho(1-\rho)(\widetilde{Z}_i^2- \Ex\widetilde{Z}^2)+(1-2\rho) (T_i-\rho)\widetilde{Z}_i^2. $$
It is obvious that
$$\frac{\sum_{i=1}^n (\widetilde{Z}_i^2- \Ex\widetilde{Z}^2)}{n}\to 0 \text{ in } L_1. $$
When $\rho\ne 1/2$,  we have
$$ \frac{\sum_{i=1}^n (T_i-\rho)\widetilde{Z}_i^2}{n}\to 0 \text{ in } L_1, $$
by (ii) and noticing the assumption that $\Ex[\widetilde{Z}^2|\bm \phi]\in Span\{\bm \phi\}$ . Hence, \eqref{eq:proof:thm4:1} holds. 

When $(Z-\Ex[Z|\bm\phi])W$ is also in $Span\{\bm \phi\}$. Then
$$\frac{ \sum_{i=1}^n (T_i-\rho)\widetilde{Z}_iW_i}{n} \to 0 \text{ in } L_1 $$
by (ii) again. Thus,
$$\sum_{i=1}^n \big((T_i-\rho)\widetilde{Z}_i,W_i\big)^{\otimes 2}\to diag\big(\vec{\sigma}_Z^2, \Ex W^2\big) \text{ in } L_1. $$
Applying the central limit theorem to martingale differences $\Big\{\big((T_i-\rho)\widetilde{Z}_i, W_i\big),\mathscr{F}_i;i=1,2,\ldots\Big\}$ yields \eqref{eq:thm:Property1:10}.
The proof is completed. $\Box$

\subsection{The properties of the tests}

{\bf Proof of Theorem \ref{thm:test}}. \eqref{eq:const} is trivial by applying \eqref{eq:thm:Property1:8ad}. For \eqref{eq:esttau}, we let $e_t=Y(t)-\mu_t$, $\tilde{e}_{t}=e_{t}-\Proj[e_{t}|\bm \phi]$,  $e_{i,t}=Y_i(t)-\mu_t$,  $D_{n,t}=\sum_{i=1}^n (T_i-\rho)e_{i,t}$,    $t=1,2$, $D_n=\sum_{i=1}^n (T_i-\rho)$. Then
$ D_n=O_p(n^{1/2})$, $D_{n,t}=O_p(n^{1/2})$, $\sum_{i=1}^n e_{i,t}=O_P(n^{1/2})$,$t=1,2$, by \eqref{eq:thm:Property1:10}. Thus, 
\begin{align*}
\widehat{\tau}_n=:&\overline{Y}_{n,1}-\overline{Y}_{n,2}=\tau+\frac{\sum_{i=1}^n e_{i,1}+\frac{D{_{n,1}}}{\rho}}{n\big(1+\frac{D_n}{n\rho}\big)}-\frac{\sum_{i=1}^n e_{i,2}-\frac{D{_{n,2}}}{1-\rho}}{n\big(1-\frac{D_n}{n(1-\rho)}\big)}\\
=&\tau+\frac{\frac{D{_{n,1}}}{\rho}+\frac{D{_{n,2}}}{1-\rho}+\sum_{i=1}^n (e_{i,1}-e_{i,2})}{n}+o_P(n^{-1/2})\\
=&\tau+\frac{\sum_{i=1}^n(T_i-\rho) \big(\frac{e{_{i,1}}}{\rho}+\frac{e{_{i,2}}}{1-\rho}\big)+\sum_{i=1}^n (e_{i,1}-e_{i,2})}{n}+o_P(n^{-1/2}).
\end{align*}
By \eqref{eq:thm:Property1:10}, it follows that
$$\sqrt{n}(\widehat{\tau}_n-\tau)\overset{d}\to N(0,\sigma_{\tau}^2) $$
with
\begin{align*}
\sigma_{\tau}^2=&\rho(1-\rho)\Ex\Big(\frac{\tilde{e}_1}{\rho}+\frac{\tilde{e}_2}{1-\rho}\Big)^2+\Ex(e_1-e_2)^2\\
=& \rho(1-\rho)\Ex\Big(\frac{e_1}{\rho}+\frac{e_2}{1-\rho}\Big)^2+\Ex(e_1-e_2)^2
-\rho(1-\rho)\Ex\Big(\Proj[\frac{e_1}{\rho}+\frac{e_2}{1-\rho}|\bm \phi]\Big)^2\\
=&  \frac{\Ex e_1^2}{\rho}+\frac{\Ex e_2^2}{1-\rho}  
-\rho(1-\rho)\Ex\Big(\Proj[\frac{e_1}{\rho}+\frac{e_2}{1-\rho}|\bm \phi]\Big)^2.
\end{align*}
Also,
\begin{align*}
\sigma_{\tau}^2=&\rho(1-\rho)\Ex\Big(\frac{\tilde{e}_1}{\rho}
+\frac{\tilde{e}_2}{1-\rho}\Big)^2+\Ex(\tilde{e}_1-\tilde{e}_2)^2
+\Ex\big(\Proj[e_1-e_2|\bm \phi]\big)^2 \\
=&\frac{\Ex\tilde{e}_1^2}{\rho}
+\frac{\Ex\tilde{e}_2^2}{1-\rho}+\Ex(\Proj[e_1-e_2|\bm \phi])^2.
\end{align*}
\eqref{eq:esttau} is proven. For \eqref{eq:test1}, it is sufficient to notice that
$$\frac{n\widehat{\sigma}_{n,1}^2}{N_{n,1}}+\frac{n\widehat{\sigma}_{n,2}^2}{N_{n,2}}
=\frac{\widehat{\sigma}_{n,1}^2}{1+\frac{D_n}{n\rho}}+\frac{\widehat{\sigma}_{n,2}^2}{1-\frac{D_n}{n(1-\rho)}}\overset{P}\to \frac{\sigma_1^2}{\rho}+\frac{\sigma_2^2}{1-\rho}. \;\;\Box$$

\bigskip

{\bf Proof of Theorem \ref{thm:regest}}.  For simplifying the proof, without loss of generality we assume that $ \Sigma=\Ex[\bm\phi^{\otimes 2}]$ is non-singular. Let $\bm \alpha_t=\Ex[\bm\phi e_t]\Sigma^{-1}$. Then $\bm\phi(\bm X)\bm \alpha_t^{\prime}=\Proj[e_t|\bm \phi]$.
By noting $\overline{Y}_{n,t}- \mu_t\overset{P}\to 0$, we have
$\widehat{\zeta}_i=Y_i-\underline{\bm X}_i\widehat{\theta}^{\prime}-\bm\phi(\bm X_i)\widehat{\bm \alpha}^{\prime}$ with
\begin{align*}
&\widehat{\bm \alpha}_{n,1}= \sum_{i=1}^n T_i\bm\phi_i(Y_i(1)-\overline{Y}_{n,1}) \left[\sum_{i=1}^nT_i\bm\phi_i^{\otimes 2}\right]^{-1}\\
=&\frac{1}{n\rho}\sum_{i=1}^n T_i\bm\phi_i(Y_i(1)-\mu_1)\Sigma^{-1}+o_P(1) 
  \overset{P}\to  \Ex\left[\bm \phi e_1\right]\Sigma^{-1}=\bm \alpha_1,
\end{align*}
by \eqref{eq:thm:Property1:8}. Thus
\begin{align*}
& \frac{1}{\sum_{i=1}^n T_i(1)-1}  \sum_{i=1}^n \big(\widehat{\zeta}_{i,1}\big)^2
=\frac{1}{N_{n,1}-1}  \sum_{i=1}^n T_i\big(Y_i(1)-\overline{Y}_{n,1}- \bm\phi_i\widehat{\bm \alpha}_{n,1}^{\prime}\big)^2\\
=&
\frac{1}{n\rho}  \sum_{i=1}^n T_i\big(e_{i,1}- \bm\phi_i \bm \alpha_1^{\prime}\big)^2+o_P(1) \\
 &\overset{P}\to    \Ex\big(e_1- \bm\phi \bm \alpha_1^{\prime}\big)^2  
 =\Ex\Big(Y(1)-\mu_1-\Proj\big[ Y(1)-\mu_1\big|\bm \phi\big]\Big)^2,
\end{align*}
by \eqref{eq:thm:Property1:8} again. 

Similarly, $\widehat{\bm \alpha}_{n,2} \overset{P}\to \bm \alpha_2$,
$$   \frac{1}{\sum_{i=1}^n T_i(2)-1}  \sum_{i=1}^n \big(\widehat{\zeta}_{i,2}\big)^2 \overset{P}\to \Ex\Big(Y(2)-\mu_2-\Proj\big[ Y(2)-\mu_2\big|\bm \phi\big]\Big)^2, $$
and
\begin{align*}
 &\frac{1}{n-2}\sum_{i=1}^n  \big(\bm \phi_i[\widehat{\bm \alpha}_{n,1}-\widehat{\bm \alpha}_{n,2}]^{\prime}\big)^2
= \frac{1}{n}\sum_{i=1}^n  \big(\bm \phi_i[ \bm \alpha_{1}- \bm \alpha_{2}]^{\prime}\big)^2+o_P(1)\\
& \overset{P}\to  \Ex\big(\bm \phi[ \bm \alpha_{1}- \bm \alpha_{2}]^{\prime}\big)^2=\Ex\Big(\Proj\big[  Y(1)-\mu_1-(Y(2)-\mu_2)\big|\bm \phi(\bm X)\big]\Big)^2.
\end{align*}
The proof is completed. $\Box$

{\bf Proof of Theorem \ref{thm:mb}}. Let $\Delta M_{n,t}=(T_n-\rho)\widetilde{e}_{n,t}- \Ex[(\ell_n-\rho)\widetilde{e}_{n,t}|\bm\Lambda_{n-1}]$, $t=1,2$. By \eqref{eq:thm:Property3:2},
$$ \sum_{i=1}^n T_ie_{i,1}
=\sum_{i=1}^n \Delta M_{i,1}+\rho \sum_{i=1}^n e_{i,1}+O(n^{\gamma/2}) \text{ in } L_r$$
for all $r>0$. It follows that
$$ \sum_{i=1}^n T_ie_{i,1}
=\sum_{i=1}^n \Delta M_{i,1}+\rho \sum_{i=1}^n e_{i,1}+o(n^{\gamma/2+\epsilon}) \; a.s. \text{ for all } \epsilon>0. $$
Thus,
$$\sum_{i=j+1}^{j+l} T_ie_{i,1}
=\sum_{i=j+1}^{j+l} \Delta M_{i,1}+\rho \sum_{i=j+1}^{j+l} e_{i,1}+o(n^{\gamma/2+\epsilon_0})\;
\text{ uniformly in } j=0,1,\ldots,n-l. $$
Since $1\in Span\{\bm \Phi(\bm X)\}$, $\widetilde{1}=1-Pj(1|\bm\phi)=0$, and so
$$\sum_{i=j+1}^{j+l}T_i=\rho l+o(n^{\gamma/2+\epsilon_0})\;
\text{ uniformly in } j=0,1,\ldots,n-l. $$
Write
 $Z_1(j,l)=\sum_{i=j+1}^{j+l} \Delta M_{i,1}/\rho+ \sum_{i=j+1}^{j+l} e_{i,1}$. Then $Z_1(j)$ is a summation of $l$ martingale differences, and
$$\overline{Y}_{1}(j,l)-\mu_1=\frac{Z_1(j,l)}{ l} \left(1+ \frac{o(n^{\gamma/2+\epsilon_0})}{l}\right)=
\frac{Z_1(j,l)}{ l} (1+o(1)) \; a.s.  $$
uniformly in $j=0,1,\ldots,n-l$, by the assumption $n^{\gamma/2+\epsilon_0}/l\to 0$. Similarly,
$$\overline{Y}_{2}(j,l)-\mu_2= \frac{Z_2(j,l)}{ l} (1+o(1)) \; a.s.  \text{ uniformly in } j=0,1,\ldots,n-l, $$
where
$Z_2(j,l)=-\sum_{i=j+1}^{j+l} \Delta M_{i,2}/(1-\rho)+ \sum_{i=j+1}^{j+l} e_{i,2}$.
It is easily seen that $\Ex (Z_t(j,l))^2\le C l$, $t=1,2$.  Thus,
$$\frac{1}{n-l}\sum_{j=0}^{n-l}(Z_t(j,l))^2=O_P(l), \;\; \frac{1}{n-l}\sum_{j=0}^{n-l}(\overline{Y}_{t}(j,l)-\mu_t)^2=O_P(l), \; t=1,2 
$$
and 
\begin{align*}
&\frac{1}{n-l+1}\sum_{j=0}^{n-l}(\overline{Y}_{1}(j,l)-\mu_1)
=\frac{1}{(n-l+1)l}\sum_{j=0}^{n-l}Z_1(j,l)+\frac{1}{\sqrt{l}}o(1)\\
=& \frac{1}{n-l+1}Z_1(0,n)+o(l^{-1/2}) 
=O(n^{-1/2})+o(l^{-1/2}) =o(l^{-1/2}) 
\end{align*}
in probability, by   the assumption $l/n\to 0$. Similarly, 
$$ \frac{1}{n-l+1}\sum_{j=0}^{n-l}(\overline{Y}_{2}(j,l)-\mu_1)=o_P(l^{-1/2}). $$ 
It follows that
\begin{align*}
   SV_n  
 =& \frac{1}{n-l}\sum_{j=0}^{n-l}\Big( \overline{Y}_{1}(j,l)-\mu_1-(\overline{Y}_{2}(j,l)-\mu_2)\Big)^2\\
&  -\frac{n-l+1}{n-l}\Big(\frac{1}{n-l+1}\sum_{j=0}^{n-l}
[\overline{Y}_{1}(j,l)-\mu_1-(\overline{Y}_{2}(j,l)-\mu_2)]
 \Big)^2\\
 =& \frac{1}{(n-l)l^2}\sum_{j=0}^{n-l} (Z_1(j,l)-Z_2(j,l))^2\\
 & +o_{a.s.}(1)\cdot \frac{1}{(n-l)l^2}\sum_{j=0}^{n-l} [(Z_1(j,l))^2+(Z_1(j,l))^2]+o_P(l^{-1}) \\
 =& \frac{1}{nl^2}\sum_{j=0}^{n-l} (Z_1(j,l)-Z_2(j,l))^2+o_P(l^{-1})=\frac{1}{nl^2} \sum_{j=0}^{n-l}\Big(\sum_{i=j+1}^{j+l}\Delta M_i\Big)^2+o_P(l^{-1}),  
\end{align*}
where
\begin{align*}
&\Delta M_i =  \Delta M_{i,1}/\rho+ \Delta M_{i,2}/(1-\rho)+ (e_{i,1}-e_{i,2})\\
=& (T_i-\rho)\Big(\frac{\widetilde{e}_{i,1}}{\rho}+\frac{\widetilde{e}_{i,2}}{1-\rho}\Big)
-\Ex\left[(\ell_n-\rho)\Big(\frac{\widetilde{e}_{i,1}}{\rho}+\frac{\widetilde{e}_{i,2}}{1-\rho}\Big)|\bm \Lambda_{n-1}\right]+  (e_{i,1}-e_{i,2}).
\end{align*} 
Notice that $\{\Delta M_i\}$ is a sequence of martingale differences with 
$$ \Ex[(\Delta M_i)^2|\mathscr{F}_{i-1}]
\to \rho(1-\rho)\Ex\left(\frac{\widetilde{e}_{i,1}}{\rho}+\frac{\widetilde{e}_{i,2}}{1-\rho}\right)^2
+\Ex(e_{i,1}-e_{i,2})^2=\sigma_{\tau}^2 \text{ in } L_1. $$
It follows that
$$ \frac{1}{nl} \sum_{j=0}^{n-l}\Big(\sum_{i=j+1}^{j+l}\Delta M_i\Big)^2
=\frac{1}{nl} \sum_{j=0}^{n-l}\sum_{i=j+1}^{j+l}\Ex[(\Delta M_i)^2|\mathscr{F}_{i-1}]+o_P(1)
\overset{P}\to \sigma_{\tau}^2 $$
c.f. the arguments of the proof of \cite[Theorem 3.5]{Zhang2023}. It follows that \eqref{eq:thm:mb:1} hold. The proof is now completed. $\Box$

\bibliographystyle{apalike}

\newpage
\bigskip

\appendix
\setcounter{equation}{0}
\renewcommand{\theequation}{A.\arabic{equation}}
\setcounter{table}{0}
\renewcommand{\thetable}{A.\arabic{table}}
\setcounter{table}{0}
\renewcommand{\thefigure}{A.\arabic{table}}

\begin{center}
 { \LARGE\bf  Covariate-Adaptive
Randomization in Clinical Trials without Inflated Variances}

 {\Large \bf Supplementary Materials}
\end{center}

\renewcommand{\thesection}{\Alph{section}}
 In the supplementary materials, we give simulation studies of the performances of the proposed CARA and the tests for treatment effects under the CAR. The simulation studies are provided by Mr. Yuhang Tao, a Phd student of the author.

\section{Simulation}

\subsection{Performances of the CAR for discrete Covariates}
In this section, we evaluate marginal imbalance and within-stratum imbalance under our proposed procedure and compare the results with simple randomization (SR) and Pocock and Simon's marginal procedure (PS) under equal ($\rho=1/2$) and unequal allocation ($\rho=2/3$) (In the case of discrete covariates, the PS procedure is a special case of the procedures  proposed by  \cite{Ma2024} for equation allocation and by \cite{Liu2025} for unequal allocation). The covariates $\boldsymbol{X}_i$ are two-dimensional, where $X_{i,1}$ takes values in $\{1,2\}$ with equal probability and $X_{i,2}$ takes in values in $\{1,2\}$ with probabilities of 0.2 and 0.8, respectively. The biased coin probability for PS is set to 0.9, consistent with \cite{Pocock1975}. For the new procedure, we consider $\gamma=0.3,0.5$ and $0.6$ with the $\ell(x)= 0.1 \vee (\rho-0.5x)\wedge 0.9$. The upper and lower bounds of $\ell(x)$ are determined by the biased coin probability of PS, in order to make the comparison more convincing. The sample sizes are $n$ = 200, 500 and 800, and each case is based on 5000 replicates. The simulation results of overall imbalance, marginal imbalance and within-stratum imbalance under equal allocation and unequal allocation are presented in Tables \ref{dis_equal} and \ref{dis_unequal}, respectively. 
 
The standard deviations of the imbalance of overall imbalance, marginal imbalance and within-stratum imbalance under the new procedure and PS procedure are substantially smaller under both the proposed procedure and the PS procedure than under SR, indicating that both procedures achieve effective covariate balance. Besides, the standard deviations under the new procedure is larger than PS and increase with $\gamma$, indicating the new procedure achieves weaker balance for discrete covariates, and its efficiency decreases as $\gamma$ increases.

\begin{table}[htbp]
    \centering
    \caption{Standard deviations of various imbalances under various randomization procedures with equal allocation ($\rho =1/2$).}
    \begin{tabular}{l*{6}{c}}
    \toprule
    Randomization & $n$ & $D_n$ & $D_n(1;1)$ & $D_n(2;1)$ & $D_n(1,1)$ & $D_n(1,2)$ \\
      \midrule
      SR  & 200 & 7.01 & 4.98 & 3.18 & 2.26 & 4.45  \\
          & 500 & 11.22 & 7.49 & 4.71 & 3.35 & 6.66\\
          & 800 & 14.36 & 10.14 & 6.30 & 4.40 & 9.15 \\
      PS   & 200 & 0.55 & 0.49 & 0.50 & 1.40 & 1.44 \\
        & 500 & 0.56 & 0.52 & 0.53 & 2.19 & 2.22 \\
          & 800 & 0.56 & 0.49 & 0.49 & 2.76 & 2.78 \\
      NEW($\gamma=0.3$) & 200 & 1.54 & 1.36 & 1.32 & 1.58 & 1.82 \\
      & 500 & 1.80 &  1.47 & 1.47 & 2.22 & 2.42  \\
         & 800& 1.92 & 1.68 & 1.62 & 3.02 & 3.21  \\
      NEW($\gamma=0.5$) & 200 & 2.53 & 2.15 & 1.98 & 1.78 & 2.23 \\
       & 500 & 3.21 & 2.59 & 2.50 & 2.47 & 3.00 \\
          & 800& 3.70 & 3.15 & 3.05 & 3.25 & 3.84  \\
      NEW($\gamma=0.6$) & 200 & 3.11 & 2.61 & 2.21 & 1.82 & 2.58\\
      & 500 & 4.28 & 3.34 & 3.04 & 2.68 & 3.45 \\
          & 800& 4.90 & 4.18 & 3.83 & 3.45 & 4.46 \\
    \bottomrule
    \end{tabular}
    \label{dis_equal}
\end{table}

\clearpage

\begin{table}[ht]
    \centering
    \caption{Standard deviations of various imbalances under various randomization procedures with unequal allocation ($\rho =2/3$).}
    \begin{tabular}{l*{6}{c}}
    \toprule
    Randomization & $n$ & $D_n$ & $D_n(1;1)$ & $D_n(2;1)$ & $D_n(1,1)$ & $D_n(1,2)$ \\
      \midrule
      SR  & 200 &  6.56 & 4.69 & 2.99 & 2.14 & 4.15  \\
          & 500 & 10.56 & 7.49 & 4.71 & 3.35 & 6.66 \\
          & 800 & 13.30 & 9.45 & 5.97 & 4.28 & 8.42 \\
      PS  & 200 & 0.61 & 0.53  & 0.52 & 1.39 & 1.42\\
          & 500 & 0.60 & 0.52  & 0.53 & 2.19 & 2.22 \\
          & 800 & 0.61 & 0.53  & 0.52 & 2.71 & 2.73 \\
      NEW($\gamma=0.3$) & 200  & 1.45 & 1.26 & 1.22 & 1.48 & 1.70 \\
          & 500 & 1.69  & 1.47 & 1.47 & 2.22 & 2.42 \\
          & 800 & 1.79  & 1.56 & 1.55 & 2.75 & 2.91 \\
      NEW($\gamma=0.5$) & 200  & 2.37 & 2.03 & 1.86 & 1.63 & 2.13 \\
          & 500 & 3.05  & 2.59 & 2.50 & 2.47 & 3.00 \\
          & 800 & 3.44  & 2.97 & 2.85 & 3.08 & 3.66 \\
      NEW($\gamma=0.6$) & 200  & 2.97 & 2.47 & 2.15 & 1.77 & 2.45\\
          & 500 & 3.99  & 3.34 & 3.04 & 2.68 & 3.45 \\
          & 800 & 4.65  & 4.03 & 3.64 & 3.28 & 4.35 \\
    \bottomrule
    \end{tabular}
    \label{dis_unequal}
\end{table}

\subsection{Performances of the CAR  for continuous Covariates}\label{simulation_cont}
In this section, we evaluate convergence rate of specified covariate feature and additional unspecified covariates under the new procedure and compare the procedure with SR and COV procedure proposed by \cite{Ma2024} for equal allocation and by \cite{Liu2025} for unequal allocation. The covariates are two-dimensional, where $X_{i,1}$ and $X_{i,2}$ are generated independently from $N(0,1)$ and $N(1,1)$, respectively.

Under equal allocation, we implement COV procedure and the NEW procedure with the feature map $\phi(X_i) = (\sqrt{w_0},\sqrt{w_1}X_i,\sqrt{w_2}\operatorname{vec}(X_iX_i'))$ and different sets of weights. The biased coin probability for COV procedure is set to 0.9, consistent with \cite{Ma2024}. For the new procedure, we consider $\gamma = 0.3, 0.5$ and $0.6$, with the allocation function specified as $\ell(x)= 0.1 \vee (\rho-0.5x)\wedge 0.9$. The sample sizes are $n$ = 200, 500 and 800, and each case is based on 5000 replicates. The simulation results are presented in Table \ref{continuous_equal}. 

Several conclusions can be drawn from Table \ref{continuous_equal}. First, standard deviations of the imbalance of covariate features with nonzero weights in the feature map under the proposed NEW procedure and COV procedure are substantially smaller  than the one under SR, indicating that both procedures achieve effective covariate balance within the targeted feature space. In addition, the standard deviations under the NEW procedure are larger than those under COV and increase with $\gamma$. This suggests that the convergence rate of the proposed procedure is slower than that of COV, and becomes slower as $\gamma$ increases, which is consistent with our theoretical findings. Besides, standard deviations of the imbalance of additional covariates vary across different settings and randomization procedure. It is worth noting that when the weights are set to $(w_0,w_1,w_2) = (0,1,1)$, standard deviations of $\sum (T_i-\rho)e^{-X_{i,2}^2}$ under the COV procedure are larger than those under SR, indicating potential variance inflation for covariates not targeted by the feature map. In contrast, under the same setting, the standard deviations under the new procedure are smaller than those under SR. This phenomenon suggests that the COV procedure may induce variance inflation for additional covariates, whereas the new procedure avoids this issue, highlighting its superior robustness in maintaining covariate balance for additional covariates.

\begin{table}[htbp]
    \centering
    \footnotesize 
    \setlength{\tabcolsep}{2.7pt} 
    \renewcommand{\arraystretch}{0.65}
    \caption{Means (standard deviations) of covariate imbalances under various randomization procedures with equal allocation ($\rho =1/2$).}
    \begin{tabular}{l*{6}{c}}
    \toprule
    Randomization & $n$ & $\sum (T_i - \rho)$ & $\sum (T_i - \rho)X_{i,1}$ & $\sum (T_i - \rho)X_{i,2}$ & $\sum (T_i - \rho)X_{i,1}^2$ & $\sum (T_i - \rho)e^{-X_{i,2}^2}$ \\
      \midrule
      SR  & 200 & 0.06(7.08) & -0.05(7.14) & -0.01(10.15) & -0.18(12.24) & 0.03(3.87) \\
          & 500 & 0.15(11.17) & -0.12(11.47) & 0.02(15.87) & 0.08(19.53) & 0.10(6.08) \\
          & 800 & -0.25(13.94) & -0.01(14.26) & -0.60(19.64) & -0.21(24.47) & -0.07(7.68) \\
      COV & 200 & 0.00(0.70) & 0.00(0.83) & 0.00(0.87) & 0.31(10.17) & -0.01(1.69)\\
            $(w_0,w_1,w_2) = (1,1,0)$ & 500 & 0.00(0.70) & 0.00(0.83) & 0.00(0.88) & -0.20(16.29) & 0.03(2.62) \\
          & 800 & 0.00(0.69) & 0.00(0.85) & -0.01(0.88) & 0.41(20.32) & -0.04(3.25)\\
      NEW($\gamma=0.3$) & 200 & 0.01(1.17) & -0.01(1.20) & -0.01(1.22) & 0.07(10.04) & -0.01(1.80)\\
            $(w_0,w_1,w_2) = (1,1,0)$ & 500 & 0.02(1.29) & 0.02(1.35) & 0.02(1.37) & -0.25(15.86) & -0.04(2.76)\\
         & 800&-0.01(1.37) & -0.02(1.43) & -0.01(1.45) & -0.39(19.96) & -0.03(3.33)  \\
      NEW($\gamma=0.5$) & 200 & 0.06(1.85) & -0.02(1.87) & -0.01(1.87) & 0.11(10.23) & 0.03(2.12)\\
           $(w_0,w_1,w_2) = (1,1,0)$  & 500 &0.05(2.34) & 0.00(2.37) & 0.00(2.31) & 0.23(16.03) & -0.05(3.09)\\
          & 800& -0.06(2.66) & -0.03(2.62) & 0.09(2.64) & -0.18(19.99) & -0.12(3.68) \\
      NEW($\gamma=0.6$) & 200 &-0.02(2.28) & 0.08(2.35) & 0.03(2.38) & 0.22(10.07) & -0.03(2.34) \\
          $(w_0,w_1,w_2) = (1,1,0)$ & 500 & -0.14(3.09) & -0.01(3.16) & -0.01(3.20) & -0.08(16.29) & -0.11(3.44) \\
          & 800&0.01(3.59) & 0.08(3.61) & -0.09(3.62) & -0.13(20.47) & 0.08(4.16) \\
      COV   & 200  & -0.06(5.53) & -0.01(1.49) & 0.00(1.56) & 0.02(2.05) & -0.05(4.06) \\
      $(w_0,w_1,w_2) = (0,1,1)$    & 500 & -0.05(8.77) & -0.02(1.49) & -0.04(1.54) & -0.02(2.05) & 0.04(6.45) \\
          & 800 & -0.03(11.02) & 0.02(1.51) & -0.01(1.59) & 0.02(2.03) & -0.07(8.24)\\
      NEW($\gamma=0.3$) & 200 & -0.07(5.15) & 0.03(1.72) & -0.01(1.78) & -0.02(2.18) & -0.03(3.77) \\
      $(w_0,w_1,w_2) = (0,1,1)$    & 500 & 0.10(7.96) & 0.04(1.81) & -0.02(1.88) & 0.03(2.20) & 0.06(5.85) \\
         & 800 &-0.02(9.90) & -0.04(1.88) & -0.01(1.95) & 0.01(2.26) & 0.06(7.32)\\
      NEW($\gamma=0.5$)    & 200 & -0.07(4.87) & -0.03(2.14) & 0.05(2.22) & 0.07(2.55) & -0.08(3.52) \\
      $(w_0,w_1,w_2) = (0,1,1)$    & 500 & 0.11(7.49) & 0.05(2.55) & 0.01(2.61) & -0.03(2.83) & 0.12(5.61) \\
          & 800 &0.26(9.45) & 0.05(2.78) & 0.01(2.87) & -0.06(3.10) & 0.19(7.09) \\
      NEW($\gamma=0.6$)   & 200& 0.02(4.90) & 0.04(2.46) & -0.03(2.49) & -0.05(2.93) & 0.00(3.57) \\
      $(w_0,w_1,w_2) = (0,1,1)$    & 500 & 0.13(7.48) & 0.08(3.24) & -0.04(3.26) & -0.01(3.52) & 0.11(5.58) \\
          & 800 & -0.04(9.23) & -0.02(3.66) & 0.05(3.68) & 0.03(4.01) & -0.03(7.00) \\
      COV   & 200 & -0.01(1.25) & 0.05(1.59) & 0.02(1.59) & 0.01(2.14) & -0.02(2.03) \\
      $(w_0,w_1,w_2) = (1,1,1)$ & 500 & 0.03(1.29) & -0.01(1.59) & 0.01(1.58) & 0.02(2.07) & -0.04(3.21)\\
          & 800 & -0.02(1.30) & -0.01(1.61) & -0.01(1.57) & 0.00(2.13) & 0.08(4.05) \\
      NEW($\gamma=0.3$) & 200 & 0.01(1.51) & -0.05(1.76) & 0.01(1.74) & 0.01(2.21) & 0.03(2.03) \\
      $(w_0,w_1,w_2) = (1,1,1)$  & 500& -0.01(1.68) & -0.01(1.88) & 0.02(1.87) & 0.00(2.28) & -0.03(3.10) \\
         & 800 & -0.03(1.72) & 0.00(1.91) & -0.03(1.89) & -0.04(2.31) & -0.11(3.75)\\
      NEW($\gamma=0.5$)    & 200 & -0.07(2.07) & 0.02(2.16) & -0.02(2.15) & 0.01(2.63) & -0.04(2.19) \\
      $(w_0,w_1,w_2) = (1,1,1)$ & 500 & 0.03(2.48) & 0.02(2.56) & 0.05(2.58) & -0.02(2.86) & 0.05(3.11) \\
          & 800 & 0.04(2.76) & 0.03(2.82) & -0.03(2.83) & -0.03(3.15) & 0.02(3.75) \\
      NEW($\gamma=0.6$)   & 200& -0.05(2.43) & -0.01(2.51) & -0.04(2.45) & -0.02(2.90) & 0.00(2.28) \\
      $(w_0,w_1,w_2) = (1,1,1)$ & 500 & -0.06(3.18) & -0.01(3.20) & 0.06(3.18) & -0.02(3.55) & -0.06(3.33)\\
          & 800& 0.08(3.64) & 0.06(3.74) & 0.01(3.59) & 0.03(4.00) & 0.02(4.09) \\
    \bottomrule
    \end{tabular}
    \label{continuous_equal}
\end{table}

Under unequal allocation, we consider a treatment assignment ratio of $\rho:(1-\rho)$ with $\rho = 2/3$. We implement COV procedure and the NEW procedure with the feature map $\phi(X_i) = (1,X_i)$. The biased coin probability for COV procedure is set to 0.9 to maintain consistency with \cite{Liu2025}. For the new procedure, we consider $\gamma = 0.3, 0.5$ and $0.6$, with $\ell_{asym}(x) = 0.1 \vee (\rho-0.5x)\wedge 0.9$ inducing an asymmetric allocation rule and $\ell_{sym}(x) = (2\rho-0.9) \vee (\rho-0.5x)\wedge 0.9$ corresponding to a symmetric allocation rule. The sample sizes are $n$ = 200, 500 and 800, and each case is based on 5000 replicates. The simulation results are presented in Table \ref{continuous_unequal}. Additionally, we simulated the means of $n^{-1}\sum_{i=1}^n (T_i-\rho)X_{i,2}^2$ and $n^{-1}\sum_{i=1}^n (T_i-\rho)e^{-X_{i,1}^2}$ for the sample size $n\in\{50,100,200,500,800,1000,2000,4000,8000\}$, which are presented in Figure \ref{fig:twoimages}. 

The standard deviations of the imbalance of covariate features with nonzero weights in the feature map exhibit the same trends and patterns as those under equal allocation, leading to identical conclusions. For the imbalance of additional covariates, under COV procedures, means of $n^{-1}\sum_{i=1}^n (T_i-\rho)X_{i,2}^2$ and $n^{-1}\sum_{i=1}^n (T_i-\rho)e^{-X_{i,1}^2}$ converge to a nonzero constant as $n$ increases, indicating the presence of "shift problem" observed in \cite{Liu2025}. In contrast, under the new procedure, the corresponding means converge to zero, which aligns with our theoretical results and demonstrates the improved stability of the proposed design. In addition, the rate at which the mean converges to zero depends on both the parameter $\gamma$ and the choice of allocation function. Specifically, a larger value of $\gamma$ leads to faster convergence to zero. Moreover, the convergence is faster under the allocation function $\ell_{sym}(x)$ than under $\ell_{asym}(x)$.


\begin{table}[htbp]
    \centering
    \footnotesize 
    \setlength{\tabcolsep}{2.7pt} 
    \renewcommand{\arraystretch}{1}
    \caption{Means (standard deviations) of covariate imbalances under various randomization procedures with unequal allocation ($\rho =2/3$).}
    \begin{tabular}{l*{6}{c}}
    \toprule
    Randomization & $n$ & $\sum (T_i - \rho)$ & $\sum (T_i - \rho)X_{i,1}$ & $\sum (T_i - \rho)X_{i,2}$ & $\sum (T_i - \rho)X_{i,2}^2$ & $\sum (T_i - \rho)e^{-X_{i,1}^2}$ \\
      \midrule
      SR  & 200 & 0.01(6.72) & 0.07(6.59) & -0.06(9.48) & -0.16(21.26) & 0.04(4.46) \\
          & 500 & 0.02(10.37) & 0.02(10.51) & 0.12(14.57) & 0.09(32.93) & 0.05(6.90) \\
          & 800 & 0.30(13.33) & -0.07(13.37) & 0.64(19.01) & 1.31(42.57) & 0.11(8.75) \\
      COV & 200 & -0.33(0.78) & 0.00(0.90) & -0.01(0.94) & 3.61(10.00) & -0.66(2.37)  \\
          & 500 & -0.32(0.74) & -0.00(0.88) & -0.01(0.94) & 9.02(15.56) & -1.26(3.69)  \\
          & 800 & -0.33(0.77) & 0.01(0.88) & 0.01(0.94) & 14.08(19.55) & -1.89(4.69)  \\
      NEW($\gamma=0.3$) & 200 & -0.10(1.15) & -0.02(1.22) & -0.21(1.25) & -3.78(10.32) & 0.64(2.38)  \\
      with $\ell_{asym}(x)$   & 500 & -0.04(1.27) & -0.00(1.35) & -0.15(1.36) & -7.56(15.31) & 1.42(3.69)  \\
         & 800 & -0.02(1.36) & 0.00(1.41) & -0.16(1.43) & -10.72(19.30) & 1.97(4.58)    \\
      NEW($\gamma=0.5$) & 200 & 0.02(1.75) & -0.01(1.80) & -0.15(1.83) & -1.90(10.27) & 0.35(2.41)  \\
      with $\ell_{asym}(x)$    & 500 & -0.00(2.18) & 0.03(2.28) & -0.04(2.25) & -2.15(15.54) & 0.32(3.73) \\
          & 800 & 0.01(2.46) & -0.02(2.51) & -0.05(2.50) & -2.74(19.76) & 0.39(4.70)   \\
      NEW($\gamma=0.6$) & 200 & 0.02(2.16) & 0.03(2.26) & 0.00(2.27) & -0.87(10.54) & 0.09(2.61)  \\
      with $\ell_{asym}(x)$     & 500 & -0.03(2.89) & 0.01(3.00) & 0.00(3.01) & -1.11(16.67) & 0.10(3.93) \\
          & 800 & -0.03(3.36) & -0.00(3.44) & 0.06(3.42) & -1.36(20.35) & 0.13(4.86) \\
     NEW($\gamma=0.3$) & 200 & 0.01(1.28) & -0.03(1.48) & -0.00(1.56) & 0.06(10.10) & -0.05(2.34)  \\
     with $\ell_{sym}(x)$     & 500 & 0.01(1.38) & -0.02(1.54) & 0.05(1.61) & 0.67(15.39) & -0.02(3.67)  \\
         & 800 &  -0.02(1.41) & -0.02(1.61) & 0.06(1.62) & 0.68(19.82) & -0.10(4.59)  \\
      NEW($\gamma=0.5$) & 200 & 0.01(1.75) & 0.02(1.91) & 0.02(1.93) & 0.16(10.29) & 0.01(2.46)  \\
      with $\ell_{sym}(x)$     & 500 & 0.03(2.22) & -0.05(2.30) & 0.05(2.31) & -0.02(16.07) & -0.07(3.80) \\
          & 800 &  -0.01(2.51) & -0.00(2.57) & 0.04(2.53) & 0.11(19.73) & 0.08(4.73) \\
      NEW($\gamma=0.6$) & 200 & 0.02(2.21) & -0.01(2.27) & 0.00(2.36) & 0.08(10.69) & -0.04(2.58)  \\
      with $\ell_{sym}(x)$     & 500 & 0.02(2.94) & -0.04(2.99) & -0.06(3.03) & 0.17(16.20) & -0.04(3.89) \\
          & 800 & -0.02(3.38) & -0.01(3.50) & -0.11(3.39) & -0.43(20.41) & 0.01(4.94) \\
    \bottomrule
    \end{tabular}
    \label{continuous_unequal}
\end{table}

\clearpage
\begin{figure}[htbp]
    \includegraphics[width=\linewidth]{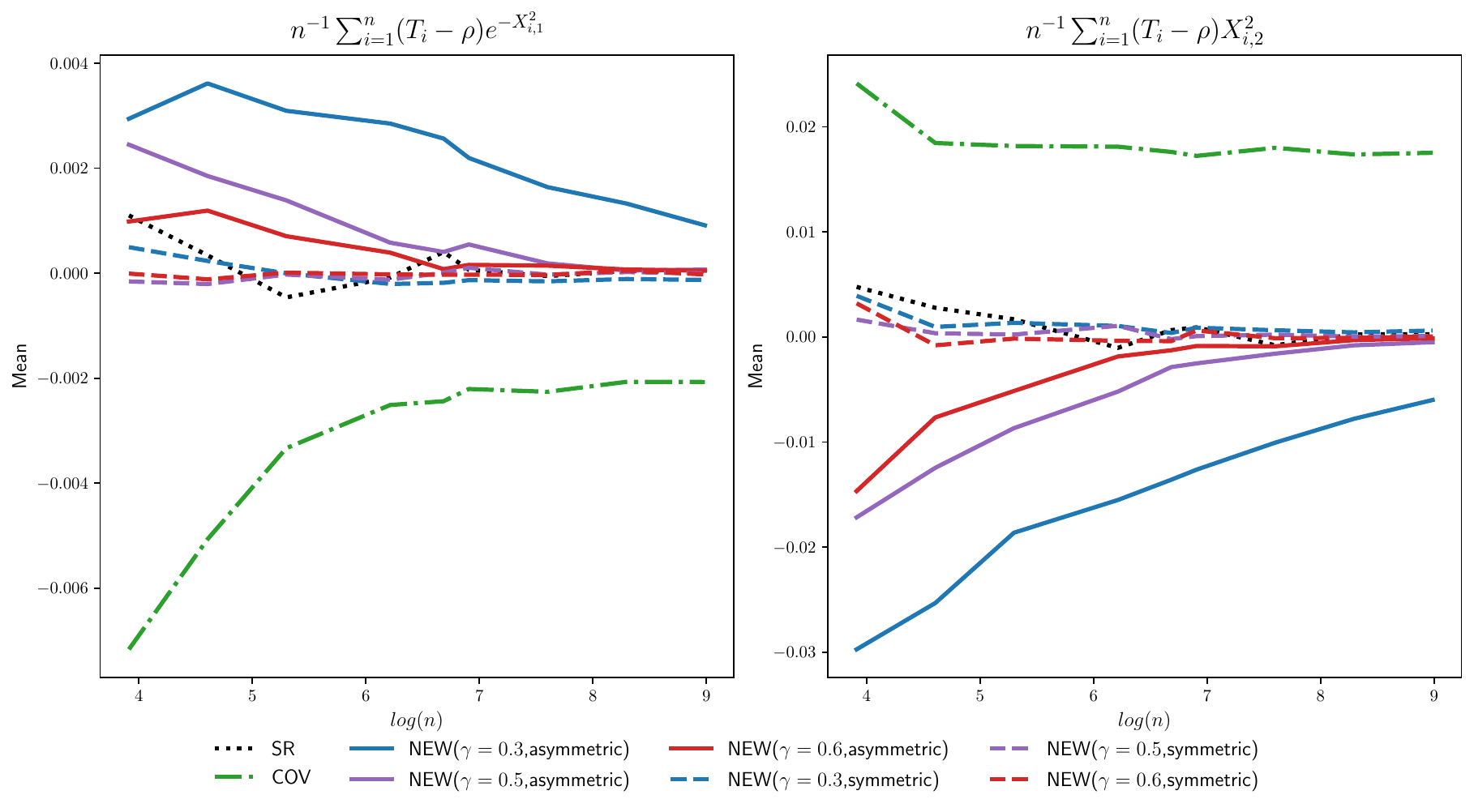}
	\caption{Simulated means of additional covariates under various sample size and randomization procedures. SR stands for the simple randomization, COV for the procedure proposed by \cite{Liu2025}, and NEW for our proposed procedure.}
	\label{fig:twoimages}
\end{figure}

\subsection{Hypothesis Testing for Treatment Effect}
In this section, we evaluated the behavior of classical test and adjusted tests for treatment effects under the new procedure. One of the adjusted test is the one based on the regression method, and is referred as Adjusted Test. The other adjusted test is the one based on the moving block estimator, and is refereed as Moving Block.   For $t\in\{1, 2\}$ and $i=1,\cdots,n$, the potential outcomes are generated by the equation $Y_i(t) = f_t(\boldsymbol{X}_i) + \sigma_t(\boldsymbol{X}_i)\varepsilon_i(t)$, where $\boldsymbol{X}_i$, $f_t(\boldsymbol{X}_i)$ and $\sigma_t(\boldsymbol{X}_i)$ are specified below. In each model, $\{\boldsymbol{X}_i, \varepsilon_i(1), \varepsilon_i(2)\}_{i=1}^n$ are i.i.d with $\varepsilon_i(1)$ and $\varepsilon_i(2)$ independently distributed as $N(0,1)$. 

Model 1: Linear model with discrete covariates. $\boldsymbol{X}_i$ is a two-dimensional discrete vector, where $X_{i,1}$ takes values in $\{-1,1\}$ with equal probability and $X_{i,2}$ takes values in $\{2,3\}$ with probabilities of 0.8 and 0.2, respectively. The two components are generated independently. The feature map is $\phi(\boldsymbol{X}_i) = (I\{X_{i,1}=-1\}, I\{X_{i,1} = 1\}, I\{X_{i,2}= 2\}, I\{X_{i,2}= 3\})$. $f_1(\boldsymbol{X}_i) = f_2(\boldsymbol{X}_i) = 2X_{i,1} + 2X_{i,2}$, $\sigma_1(\boldsymbol{X}_i) =\sigma_2(\boldsymbol{X}_i) = 2 e^{X_{i,2}-X_{i,1}-2}$. 

Model 2: Nonlinear model with discrete covariates. The specification of $\boldsymbol{X}_i$ and the feature map $\phi(\boldsymbol{X}_i)$ are the same as those in Model 1. $f_1(\boldsymbol{X}_i) = 2X_{i,1} + X_{i,2} + e^{X_{i,2}-X_{i,1}-2}$, $f_2(\boldsymbol{X}_i) = X_{i,1} + X_{i,2} + e^{X_{i,2}+X_{i,1}-2}$, $\sigma_1(\boldsymbol{X}_i) = 2 + X_{i,1}^2$, $\sigma_2(\boldsymbol{X}_i) = 1 + X_{i,2}^2$.

Model 3: Linear model with continuous covariates. $\boldsymbol{X}_i$ is a two-dimensional continuous vector, where $X_{i,1}\sim N(0,1)$, $X_{i,2}\sim N(1,1)$ and they are independent. The feature map is $\phi(\boldsymbol{X}_i) = (1,  X_{i,1}, X_{i,2})$. The specification of $f_1(\boldsymbol{X}_i)$, $f_2(\boldsymbol{X}_i)$, $\sigma_1(\boldsymbol{X}_i)$ and $\sigma_2(\boldsymbol{X}_i)$ are the same as those in Model 1.

Model 4: Nonlinear model with continuous covariates. The specification of $\boldsymbol{X}_i$ and the feature map $\phi(\boldsymbol{X}_i)$ are the same as those in Model 3 and the specification of $f_1(\boldsymbol{X}_i)$, $f_2(\boldsymbol{X}_i)$, $\sigma_1(\boldsymbol{X}_i)$ and $\sigma_2(\boldsymbol{X}_i)$ are the same as those in Model 2. 

For each model, we implement the new procedure under equal allocation ($\rho = 1/2$) and unequal allocation ($\rho = 2/3$), respectively. All remaining parameters are set the same as those in Section \ref{simulation_cont}. We simulate data with $\tau = 0$ and $\tau = 1$, respectively, to obtain the Type I errors and powers of the classical tests, adjusted tests  based  on the regression estimator and the tests based on moving block estimator, where $\ell = [\sqrt{n}]$ for the moving block estimator. The sample size is $n=500$ and the results are based on 20,000 replicates. The simulation results under equal allocation and unequal allocation are presented in Table \ref{hypothesis testing_equal} and Table \ref{hypothesis testing_unequal}, respectively.   

Several conclusions can be drawn from Tables \ref{hypothesis testing_equal} and \ref{hypothesis testing_unequal}. First, under various settings, the empirical Type I errors are all below the nominal 5\% level, indicating that the proposed procedure leads to conservative results for classical tests of the treatment effects. This observation aligns with our theoretical findings. Second, all Type I errors of adjusted tests of the treatment effect are well controlled at the nominal 5\% level. Third, the test based on the moving block estimator remains conservative, but is less conservative compared with the classical test. Moreover, compared with the traditional tests, the adjusted tests exhibit substantially higher power, demonstrating the effectiveness of the proposed adjusted test statistics. Finally, the test based on the moving block estimator generally achieves higher power than the classical test, but remains less powerful than the adjusted test. Therefore, it can serve as a useful alternative to the classical test when $\phi(\bm X_i)$ is unavailable.

\begin{table}[htbp]
    \centering
    \footnotesize 
    \setlength{\tabcolsep}{3pt} 
    \renewcommand{\arraystretch}{1}
    \caption{Type I errors and powers ($\tau = 1$) of classical test, adjusted test, test based on moving block estimator for treatment effect under the new CAR procedure in various settings with equal allocation}
    \begin{tabular}{*{8}{c}}
    \toprule
     & & \multicolumn{2}{c}{Classical Test} & \multicolumn{2}{c}{Adjusted Test} & \multicolumn{2}{c}{Moving Block}  \\
     \cmidrule(lr{0pt}){3-4} \cmidrule(lr{0pt}){5-6} \cmidrule(lr{0pt}){7-8} 
      Model & $\gamma$ & Type I  & Power & Type I  & Power & Type I  & Power  \\
      \midrule
    1    & 0.3 & 0.037 & 0.438 & 0.053 & 0.488 & 0.048 & 0.449 \\
         & 0.5 & 0.039 & 0.438 & 0.053 & 0.491 & 0.043 & 0.444 \\
         & 0.6 & 0.040 & 0.439 & 0.054 & 0.491 & 0.052 & 0.433 \\
    2    & 0.3 & 0.035 & 0.521 & 0.050 & 0.584 & 0.047 & 0.568 \\
         & 0.5 & 0.037 & 0.531 & 0.053 & 0.591 & 0.038 & 0.541 \\
         & 0.6 & 0.036 & 0.525 & 0.052 & 0.585 & 0.042 & 0.533 \\
    3    & 0.3 & 0.014 & 0.577 & 0.046 & 0.690 & 0.028 & 0.619 \\
         & 0.5 & 0.015 & 0.570 & 0.047 & 0.681 & 0.023 & 0.595 \\
         & 0.6 & 0.015 & 0.575 & 0.050 & 0.683 & 0.025 & 0.572 \\
    4    & 0.3 & 0.022 & 0.671 & 0.052 & 0.776 & 0.038 & 0.708 \\
         & 0.5 & 0.022 & 0.671 & 0.050 & 0.778 & 0.030 & 0.672 \\
         & 0.6 & 0.024 & 0.666 & 0.054 & 0.775 & 0.029 & 0.661 \\
    \bottomrule
    \end{tabular}
    \label{hypothesis testing_equal}
\end{table}

\clearpage
\begin{table}[ht]
    \centering
    \footnotesize 
    \setlength{\tabcolsep}{3pt} 
    \renewcommand{\arraystretch}{1}
    \caption{Type I errors and powers ($\tau = 1$) of classical test, adjusted test, test based on moving block estimator for treatment effect under the new CAR procedure in various settings with unequal allocation ($\rho =2/3$)}
    \begin{tabular}{*{14}{c}}
    \toprule
     & & \multicolumn{6}{c}{$\rho = 2/3$ with $\ell_{asym}(x)$} & \multicolumn{6}{c}{$\rho = 2/3$ with $\ell_{sym}(x)$} \\
    \cmidrule(lr{0pt}){3-8} \cmidrule(lr{0pt}){9-14}
     & & \multicolumn{2}{c}{Classical Test} & \multicolumn{2}{c}{Adjusted Test} & \multicolumn{2}{c}{Moving Block} & \multicolumn{2}{c}{Classical Test} & \multicolumn{2}{c}{Adjusted Test}& \multicolumn{2}{c}{Moving Block} \\
     \cmidrule(lr{0pt}){3-4} \cmidrule(lr{0pt}){5-6} \cmidrule(lr{0pt}){7-8} \cmidrule(lr{0pt}){9-10} \cmidrule(lr{0pt}){11-12} \cmidrule(lr{0pt}){13-14}
      Model & $\gamma$ & Type I  & Power & Type I  & Power & Type I  & Power & Type I  & Power & Type I  & Power & Type I  & Power \\
      \midrule
    1    & 0.3 &  0.037 & 0.393 & 0.052 & 0.444 &0.048 &0.421 &  0.034 & 0.396 & 0.050 & 0.446 & 0.047 & 0.403 \\
         & 0.5 &  0.038 & 0.396 & 0.052 & 0.447 &0.042 &0.379 &  0.038 & 0.391 & 0.053 & 0.444 & 0.043 & 0.390 \\
         & 0.6 &  0.040 & 0.397 & 0.055 & 0.445 &0.042 &0.367 &  0.039 & 0.401 & 0.054 & 0.450 & 0.041 & 0.385 \\
    2    & 0.3 &  0.036 & 0.408 & 0.054 & 0.470 &0.047 &0.421 &  0.036 & 0.414 & 0.055 & 0.475 & 0.046 & 0.409 \\
         & 0.5 &  0.036 & 0.407 & 0.052 & 0.471 &0.037 &0.394 &  0.036 & 0.409 & 0.056 & 0.473 & 0.040 & 0.400 \\
         & 0.6 &  0.040 & 0.415 & 0.057 & 0.477 &0.040 &0.403 &  0.038 & 0.410 & 0.057 & 0.471 & 0.042 & 0.401 \\
    3    & 0.3 &  0.012 & 0.529 & 0.044 & 0.650 &0.020 &0.568 &  0.012 & 0.552 & 0.044 & 0.671 & 0.021 & 0.570 \\
         & 0.5 &  0.014 & 0.531 & 0.048 & 0.651 &0.022 &0.537 &  0.015 & 0.547 & 0.049 & 0.669 & 0.021 & 0.546 \\
         & 0.6 &  0.016 & 0.543 & 0.051 & 0.660 &0.015 &0.526 &  0.015 & 0.542 & 0.051 & 0.658 & 0.020 & 0.536 \\
    4    & 0.3 &  0.021 & 0.545 & 0.053 & 0.667 &0.029 &0.556 &  0.022 & 0.604 & 0.052 & 0.718 & 0.030 & 0.625 \\
         & 0.5 &  0.023 & 0.575 & 0.057 & 0.696 &0.030 &0.576 &  0.024 & 0.603 & 0.055 & 0.719 & 0.027 & 0.598 \\
         & 0.6 &  0.024 & 0.583 & 0.057 & 0.704 &0.028 &0.581 &  0.024 & 0.597 & 0.056 & 0.714 & 0.026 & 0.602 \\
    \bottomrule
    \end{tabular}
    \label{hypothesis testing_unequal}
\end{table}
\end{document}